\documentclass[10.5pt]{amsart}
\usepackage{xcolor}
\usepackage{hyperref}
\hypersetup{
	colorlinks=true,
	citecolor=blue,
	linkcolor=blue,
	filecolor=magenta,      
	urlcolor=cyan,
}
\usepackage[capitalise,noabbrev]{cleveref}

\usepackage{amsmath,amscd}
\usepackage{amssymb}
\usepackage{mathtools}
\usepackage{stmaryrd}
\usepackage{url}
\usepackage{tikz-cd}
\usepackage{enumitem}
\usepackage{fullpage}
\usepackage{setspace}
\usepackage{microtype}

\usepackage{tikz}

\usepackage{enumitem,kantlipsum}

\author{Isaac Bird}
\author{Liran Shaul}
\author{Prashanth Sridhar}
\author{Jordan Williamson}

\address{Department of Algebra, Faculty of Mathematics and Physics, Charles University in Prague, Sokolovsk\'a 83, 186 75 Praha, Czech Republic}

\email{bird@karlin.mff.cuni.cz}
\email{shaul@karlin.mff.cuni.cz}
\email{sridhar@karlin.mff.cuni.cz}
\email{williamson@karlin.mff.cuni.cz}
\newtheorem{thm}[equation]{Theorem}
\newtheorem*{thm*}{Theorem}
\newtheorem*{cor*}{Corollary}
\newtheorem*{dfn*}{Definition}

\newtheorem{cthm}{Theorem}

\newtheorem{cor}[equation]{Corollary}
\newtheorem{prop}[equation]{Proposition}
\newtheorem{lem}[equation]{Lemma}

\theoremstyle{definition}
\newtheorem{dfn}[equation]{Definition}
\newtheorem{rem}[equation]{Remark}
\newtheorem{exa}[equation]{Example}

\newcommand{\opn}{\operatorname}
\newcommand{\cat}[1]{\operatorname{\mathsf{#1}}}

\newcommand{\mfrak}[1]{\mathfrak{#1}}

\newcommand{\msf}[1]{\mathsf{#1}}

\newcommand{\mrm}[1]{\mathrm{#1}}
\newcommand{\mbb}[1]{\mathbb{#1}}

\newcommand{\m}{\mfrak{m}}
\newcommand{\n}{\mfrak{n}}
\newcommand{\p}{\mfrak{p}}
\newcommand{\q}{\mfrak{q}}
\newcommand{\injdim}{\operatorname{inj\,dim}}
\newcommand{\projdim}{\operatorname{proj\,dim}}
\newcommand{\flatdim}{\operatorname{flat\,dim}}
\newcommand{\seqdepth}{\operatorname{seq.depth}}

\newcommand{\amp}{\operatorname{amp}}
\newcommand{\Hom}{\operatorname{Hom}}
\newcommand{\RHom}{\msf{R}\mrm{Hom}}
\newcommand{\Lotimes}{\otimes^\msf{L}}
\def\skewtimes{\ltimes\!}
\newcommand{\D}{\msf{D}}

\begin{document}
\title{Finitistic dimensions over commutative DG-rings}

\begin{abstract}
In this paper we study the finitistic dimensions of commutative noetherian non-positive DG-rings with finite amplitude. We prove that any DG-module $M$ of finite flat dimension over such a DG-ring satisfies $\projdim_A(M) \leq \dim(\mrm{H}^0 (A)) - \inf(M)$. We further provide explicit constructions of DG-modules with prescribed projective dimension and deduce that the big finitistic projective dimension satisfies the bounds $\dim(\mrm{H}^0 (A)) - \amp(A) \leq \mathsf{FPD}(A) \leq \dim(\mrm{H}^0(A))$. Moreover, we prove that DG-rings exist which achieve either bound. As a direct application, we prove new vanishing results for the derived Hochschild (co)homology of homologically smooth algebras.
\end{abstract}

\numberwithin{equation}{section}

\maketitle
\setcounter{tocdepth}{1}
\tableofcontents

\section{Introduction}
Recent years have seen an explosion in the use of differential graded (DG) techniques in commutative algebra. The study of DG-rings encompasses the affine theory of derived algebraic geometry, but also feeds back into commutative algebra via the viewpoint that DG-rings are resolutions of ordinary rings. In this paper, we continue the study of DG-rings via their homological dimensions, in particular, finding bounds on projective dimensions and thus on the finitistic dimensions of such DG-rings.

Of the homological invariants of modules, projective dimension is classically the one of most interest, both in terms of understanding the structure of the module category, and in attaching a numerical invariant to a ring (the global dimension). In general, however, projective dimension is not adequate to understand the structure of a module category, nor provide a meaningful ring invariant. For example, the Auslander-Buchsbaum-Serre theorem shows that the global dimension of a commutative noetherian ring of finite Krull dimension is finite if and only if the ring is regular, and thus there are commutative rings whose module categories have extremely simple structures but infinite global dimension.

Therefore, instead one seeks an alternative numerical invariant which captures the complexity of rings. One can study the subcategory of modules that have finite projective dimension, and from this one may define the big (resp., small) finitistic projective dimension of the ring to be the supremum of the projective dimensions of all modules (resp., finitely generated modules) of finite projective dimension. Bass attached two questions to these invariants: \begin{enumerate}
\item do the big and small finitistic projective dimensions coincide?
\item are these finitistic dimensions finite?
\end{enumerate}
In the case of commutative noetherian rings, in general both questions have definitive negative answers. The Auslander-Buchsbaum formula shows that over a commutative noetherian local ring, the small finitistic projective dimension is equal to the depth of the ring. The big finitistic projective dimension of a commutative noetherian ring is equal to its Krull dimension, due to results of Bass~\cite{Bass} and subsequently Raynaud--Gruson~\cite{RG}. In particular, we see that for commutative noetherian local rings, the two finitistic projective dimensions coincide if and only if the ring is Cohen-Macaulay, and that the big finitistic projective dimension is finite if and only if the Krull dimension is finite.

In contrast with the commutative case,
the study of the finitistic projective dimension is a very active area of research in noncommutative algebra,
and problems related to the finitistic dimensions of finite dimensional algebras are among the most important open problems in noncommutative algebra. In particular, the answer to the second question above is unknown in this setting. See the survey \cite{Hu} for details.

In this paper, we answer the analogous questions for commutative noetherian DG-rings. Throughout we work with non-positively graded commutative noetherian DG-rings of finite amplitude (in cohomological notation). We note that by \cite[Theorem 0.2]{JorAmp} and \cite[Theorem 7.21]{YeDual}, such DG-rings never have finite global dimension. Our first main result, which appears as \cref{thm:globalfpdbound} in the main body, provides an upper bound on the projective dimension of any bounded DG-module of finite flat dimension.
\begin{cthm}\label{firsttheorem}
Let $A$ be a commutative noetherian DG-ring with bounded cohomology. If $M\in\mathsf{D}^{\mrm{b}}(A)$ has $\flatdim_{A}(M)<\infty$, then $\projdim_{A}(M)\leq \dim(\mrm{H}^{0}(A))-\inf(M)$.
\end{cthm}

The proof of the above theorem is the cumulation of Sections \ref{sec:extbounds}-\ref{sec:global}. The proof builds on various amplitude inequalities proved in \cref{sec:extbounds}, which are of independent interest. These enable us to prove a special case of \cref{firsttheorem}, namely in the case when the DG-ring at hand is local and admits a dualizing DG-module, see \cref{thm:upperBoundamp}. Since the derived completion of a DG-ring admits such a dualizing DG-module, we can use faithfully flat descent to pass to the remaining part of the proof, which is the passage from local to global, see \cref{sec:global}.

We make an important further comment: those familiar with projective dimensions of complexes over rings may wonder whether the same result for complexes combined with a reduction argument will provide a proof of the above theorem. This method does not work since one cannot control the infimum of the reduction in general, as discussed in more length in \cref{rem:reduction-bound}.

Having obtained an upper bound on the projective dimension of DG-modules of finite projective dimension, we turn our attention to constructing DG-modules of prescribed projective dimension. This is our second main result and appears in the text as \cref{thm:construction}.

\begin{cthm}
Let $A$ be a commutative noetherian DG-ring with bounded cohomology such that $d=\dim(\mrm{H}^{0}(A))\leq\infty$. Then for any $0\leq n\leq d$ there is an $M\in\mathsf{D}^{\mrm{b}}(A)$ with $\sup(M)=0$, $\inf(M)\geq \inf(A)$ and $\projdim_{A}(M)=n$.
\end{cthm}
In particular, we are able to find a DG-module $M$ with $\projdim_{A}(M)=\dim(\mrm{H}^{0}(A))$ and $\inf(M)\geq -\amp(A)$. The proof of this theorem builds on results of Shaul regarding the Cohen-Macaulay loci of DG-rings~\cite{ShCM, ShLoci}.

Using the above two theorems, we can tackle Bass's questions regarding finitistic projective dimension in the DG-setting. As with complexes, the projective dimension of a DG-module can be arbitrarily large, and thus we must normalize the definition of the finitistic projective dimension. To this end we consider the number
\[
\mathsf{FPD}(A)=\sup\{\projdim_{A}(M)+\inf(M)\mid M\in\mathsf{D}^{\mrm{b}}(A) \text{ such that }\projdim_{A}(M)<\infty\},
\]
which encompasses the classical case when $A$ is a discrete ring. The following result is the main result of our paper in relation to finitistic dimensions, and appears as \cref{cor:FPDbounds} in the body of the text.

\begin{cthm}
Let $A$ be a commutative noetherian DG-ring with bounded cohomology such that $\dim(\mrm{H}^{0}(A))<\infty$. Then there are inequalities
\[
\dim(\mrm{H}^{0}(A))-\amp(A)\leq \mathsf{FPD}(A)\leq \dim(\mrm{H}^{0}(A)).
\]
\end{cthm}
Note that if $A$ is an ordinary ring (i.e., a DG-ring concentrated in one degree), then the above theorem encompasses the classic results of Bass~\cite{Bass} and Raynaud--Gruson~\cite{RG}. However, this is where the similarity to the world of ordinary rings ends, for the inequalities cannot be improved upon. Infact, in \cref{thm:examples} we show that it is possible for any $d\geq 0$ and $n>0$, to construct commutative noetherian DG-rings $A$ and $B$ such that $\amp(A)=\amp(B)=n$ and $\dim(\mrm{H}^{0}(A))=\dim(\mrm{H}^{0}(B))=d$, but $\mathsf{FPD}(A)=d-n$ and $\mathsf{FPD}(B)=d$. Thus either bound can be achieved in the DG-setting, and moreover, this dichotomy occurs for any dimension and amplitude.

The above theorem also enables us to consider the first of Bass's questions in the DG-world: when is $\mathsf{fpd}(A)=\mathsf{FPD}(A)$? This appears in~\cref{sec:fpd} of the body of the paper. In the local case, we may use the DG-version of the Auslander-Buchsbaum formula to deduce, as in~\cref{smallfindimdepth}, that
\[
\mathsf{fpd}(A)=\seqdepth_{A}(A)-\amp(A).
\]
and thus a commutative noetherian local DG-ring with bounded cohomology is local-Cohen-Macaulay if and only if $\mathsf{fpd}(A)=\dim(\mrm{H}^{0}(A))-\amp(A)$. In particular, we see that in the world of commutative noetherian local DG-rings a direct analogue of the answer for Bass's first question is not forthcoming: it remains an open question whether $\mathsf{FPD}(A)=\mathsf{fpd}(A)$ if and only if $A$ is local-Cohen-Macaulay. 

    In \cref{sec:fpd} we also consider the DG-versions of the finitistic flat and finitistic injective dimension, and consider how they relate to each other and to the finitistic projective dimension discussed above. In \cref{sec:examples}, alongside the aforementioned \cref{thm:examples}, we consider some examples. In particular, we show that for Gorenstein local DG-rings $A$ with $\mrm{H}^0(A)$ also Gorenstein, one obtains sharper bounds in \cref{firsttheorem}, and as such $\msf{FPD}(A) = \mrm{dim}(\mrm{H}^0(A)) - \amp(A)$. Finally in \cref{sec:smooth}, we consider an application to homologically smooth maps, as introduced by Kontsevich, giving new vanishing results for the derived Hochschild homology and cohomology.

\subsection*{Acknowledgements}
The second-named author thanks Jan Trlifaj for helpful discussions.
The authors thank the referees for various suggestions that helped improving the paper.
Isaac Bird, Liran Shaul and Jordan Williamson were supported by the grant GA~\v{C}R 20-02760Y from the Czech Science Foundation.
Prashanth Sridhar was supported by the grant GA~\v{C}R 20-13778S from the Czech Science Foundation. 

\section{Homological dimensions over DG-rings}\label{sec:preliminaries}
\subsection{Setup and conventions}
Throughout we will be working with commutative noetherian DG-rings. Recall that a DG-ring is a graded ring $A$ with a differential $d\colon A \to A$ of codegree 1 which satisfies the Leibniz rule. We say that $A$ is commutative if it is graded-commutative and moreover $a^2 = 0$ if $a$ has odd codegree. In this paper, all DG-rings will be assumed to be non-positive meaning that $A^i = 0$ for $i > 0$. We refer the reader to~\cite{Yebook} for more information regarding DG-rings.

A DG-module $M$ over $A$ is a graded $A$-module together with a differential of codegree 1 satisfying the Leibniz rule. The DG-$A$-modules form an abelian category, and by inverting the quasi-isomorphisms we obtain the derived category $\msf{D}(A)$ which is triangulated. When $A$ is commutative, the derived category $\msf{D}(A)$ is moreover tensor-triangulated, and we write $- \Lotimes_A -$ for its monoidal product, and $\RHom_A(-,-)$ for the internal hom. Due to our cohomological notation, we denote $\mrm{Ext}^i_A(M,N) = \mrm{H}^i(\RHom_A(M,N))$ and $\mrm{Tor}_i^A(M,N) = \mrm{H}^{-i}(M \Lotimes_A N)$.

For any $M \in \msf{D}(A)$, we define \[\sup(M) = \sup\{i \mid \mrm{H}^i(M) \neq 0\} \quad \mrm{and} \quad \inf(M) = \inf\{i \mid \mrm{H}^i(M) \neq 0\}.\] 
We write $\msf{D}^{+}(A)$ (resp., $\msf{D}^{-}(A)$) for the full subcategory of $\msf{D}(A)$ consisting of the DG-modules $M$ with $\inf(M) < \infty$ (resp., $\sup(M) < \infty$). We also define $\amp(M) = \sup(M) - \inf(M)$ and write $\msf{D}^\mrm{b}(A)$ for the full subcategory of $\msf{D}(A)$ on those DG-modules $M$ with $\amp(M) < \infty$. We denote the full subcategory of $M \in \msf{D}(A)$ for which each $\mrm{H}^i(M)$ is a finitely generated $\mrm{H}^0(A)$-module by $\msf{D}_\mrm{f}(A)$. We define $\msf{D}^\mrm{b}_\mrm{f}(A) = \msf{D}^\mrm{b}(A) \cap \msf{D}_\mrm{f}(A)$ and similarly for $\msf{D}^+_\mrm{f}(A)$ and $\msf{D}^-_\mrm{f}(A)$. We also write $\msf{D}^0(A)$ for the full subcategory of $\msf{D}(A)$ consisting of those DG-modules $M$ with $\amp(M) = 0$.

If $A$ is a commutative DG-ring, then $\mrm{H}^0(A)$ is itself a commutative ring and there is a canonical map $A \to \mrm{H}^0(A)$ of DG-rings. We say that $A$ is noetherian if $\mrm{H}^0(A)$ is noetherian and for all $i < 0$, the $\mrm{H}^0(A)$-module $\mrm{H}^i(A)$ is finitely generated. Ideals of $\mrm{H}^{0}(A)$ are denoted by $\bar{\mfrak{a}}$, following the conventions of \cite{YeDual}. If $\mrm{H}^0(A)$ is a local ring with maximal ideal $\bar\m$, then we say that $(A, \bar\m)$ is a local DG-ring. 

The canonical map $A\to\mrm{H}^{0}(A)$ induces two functors $\msf{D}(A)\to\msf{D}(\mrm{H}^{0}(A))$, namely reduction and coreduction. These are given, respectively, by
\[
\begin{matrix}
\mrm{H}^{0}(A)\Lotimes_A- & \mbox{ and } & \RHom_A(\mrm{H}^{0}(A),-).
\end{matrix}
\]
Both functors play a crucial role throughout our study.

\subsection{Homological dimensions}
We now recall the central players of our study, which are homological dimensions of DG-modules over DG-rings. These definitions extend those made in~\cite{AF} for unbounded complexes over discrete rings to the DG-setting.
In the DG setting, they were studied in ~\cite{APThesis}.

\begin{dfn}\label{dfn:homDim}
	Let $A$ be a DG-ring and $M \in \msf{D}(A)$.
	\begin{itemize}
		\item[(i)] The \emph{injective dimension} of $M$ is defined by
		\[\injdim_A(M) = \inf\{n \in \mbb{Z} \mid \mrm{Ext}_A^i(N,M) = 0 \text{ for any $N \in \msf{D}^\mrm{b}(A)$ and any $i > n - \inf(N)$}\}.\]
		\item[(ii)] The \emph{flat dimension} of $M$ is defined by
		\[\flatdim_A(M) = \inf\{n \in \mbb{Z} \mid \mrm{Tor}^A_i(N,M) = 0 \text{ for any $N \in \msf{D}^\mrm{b}(A)$ and any $i > n - \inf(N)$}\}.\]
		\item[(iii)] The \emph{projective dimension} of $M$ is defined by
		\[\projdim_A(M) = \inf\{n \in \mbb{Z} \mid \mrm{Ext}_A^i(M,N) = 0 \text{ for any $N \in \msf{D}^\mrm{b}(A)$ and any $i > n + \sup(N)$}\}.\]
	\end{itemize}
\end{dfn}

It is sufficient to test these against modules (i.e., DG-modules with amplitude 0) by a simple inductive argument as we now show.
\begin{prop}\label{prop:zeroAmp}
	Let $A$ be a DG-ring and $M \in \msf{D}(A)$. Then we have:
	\begin{itemize}
		\item[(i)] $\injdim_A(M) = \inf\{n \in \mbb{Z} \mid \mrm{Ext}_A^i(N,M) = 0 \text{ for any $N \in \msf{D}^0(A)$ and any $i > n - \inf(N)$}\}$;
		\item[(ii)] $\flatdim_A(M) = \inf\{n \in \mbb{Z} \mid \mrm{Tor}^A_i(N,M) = 0 \text{ for any $N \in \msf{D}^0(A)$ and any $i > n - \inf(N)$}\}$;
		\item[(iii)] $\projdim_A(M) = \inf\{n \in \mbb{Z} \mid \mrm{Ext}_A^i(M,N) = 0 \text{ for any $N \in \msf{D}^0(A)$ and any $i > n + \sup(N)$}\}$.
	\end{itemize}
\end{prop}
\begin{proof}
	Since the proofs are similar for all three cases, we prove only the injective dimension case. We must show that if $\mrm{Ext}^i_A(N,M)=0$ for all $i>n-\inf(N)$ where $\amp(N)=0$ then $\mrm{Ext}_A^i(N,M) = 0$ for all $N \in \msf{D}^\mrm{b}(A)$ and $i>n-\inf(N)$. We argue by induction on $\amp(N)=0$. The base case is clear, so suppose the claim holds whenever $\amp(N) < k$. 
	
	Fix an $N$ with $\amp(N) = k$. By taking truncations, there is a triangle \[N' \to N \to N''\] where $\amp(N')<k$, $\amp(N'')<k$, $\inf(N')\geq \inf(N)$ and $\inf(N'')\geq \inf(N)$. Applying $\RHom_A(-,M)$ to the above triangle and applying the long exact sequence in cohomology we obtain
	\[\cdots \to \mrm{Ext}_A^i(N'',M) \to \mrm{Ext}_A^i(N,M) \to \mrm{Ext}_A^i(N',M) \to \cdots.\] 
	Suppose $i > n - \inf(N)$.  By the properties of $N'$ and $N''$ as above, it follows that $i>n-\inf(N')$ and $i>n-\inf(N'')$ and hence by inductive hypothesis we are done.
\end{proof}
We note that an alternative proof of the projective case can be found at \cite[Theorem 2.22]{Mi}. As mentioned above, the homological dimensions behave very well with respect to reduction and coreduction, as the following corollary illustrates.

\begin{cor}\label{cor:injdimcoreduction}
	Let $A$ be a DG-ring and $M \in \msf{D}(A)$. Then we have:
	\begin{itemize}
		\item[(i)] $\projdim_A(M) = \projdim_{\mrm{H}^0(A)}(\mrm{H}^0(A)\Lotimes_A M)$;
		\item[(ii)] $\flatdim_A(M) = \flatdim_{\mrm{H}^0(A)}(\mrm{H}^0(A) \Lotimes_A M)$;
		\item[(iii)] $\injdim_A(M) = \injdim_{\mrm{H}^0(A)}(\RHom_A(\mrm{H}^0(A),M))$.
	\end{itemize}
\end{cor}
\begin{proof}
	This follows from adjunction and \cref{prop:zeroAmp}.
\end{proof}

We also note the following, which follows immediately from the definitions.
\begin{lem}\label{lem:pdshift}
	Let $A$ be a DG-ring.
	\begin{itemize}
		\item[(i)] For any $M \in \msf{D}(A)$ and $n \in \mathbb{Z}$ we have $\projdim_A(M[n]) = \projdim_A(M) + n$.
		\item[(ii)] For any $M,N \in \cat{D}(A)$ we have $\projdim_A(M\oplus N) = \max(\projdim_A(M),\projdim_A(N)).$
	\end{itemize}
\end{lem}

A natural question is to wonder which objects in $\msf{D}(A)$ correspond to the projective (or flat, or injective) modules over $\mrm{H}^{0}(A)$. We let
\[
\mrm{Proj}(A)=\{M\in\msf{D}(A): M=0 \text{ or }\projdim_A(M)=0=\sup(M)\}.
\]
The crucial fact about $\mrm{Proj}(A)$, as illustrated in \cite[Lemma 2.8]{Mi}, is that $\mrm{H}^{0}(-):\mrm{Proj}(A)\to\mrm{Proj}(\mrm{H}^{0}(A))$ is an equivalence of categories. In particular, $\mrm{Proj}(A)$ is closed under coproducts, extensions, and summands. Combining these facts, the following lemma is trivial.

\begin{lem}\label{prop:dgSplit}
		Let $A$ be a DG-ring,
		and let $P \in \opn{Proj}(A)$.
		Suppose we are given a distinguished triangle
		\[
		N \xrightarrow{\alpha} P \xrightarrow{\beta} M \to N[1]
		\]
		in $\cat{D}(A)$,
		and suppose that $\sup(M) = 0$
		and $\opn{Ext}^1_A(M,N) = 0$.
		Then $M \in \opn{Proj}(A)$.
\end{lem}

We similarly let $\mrm{Flat}(A)=\{M\in\msf{D}(A): M=0 \text{ or }\flatdim_A(M)=0=\sup(M)\}.$
If $F\in\mrm{Flat}(A)$ then $\mrm{H}^0(F)$ is a flat $\mrm{H}^0(A)$-module and
\begin{equation}\label{eqn:flat}
	\mrm{H}^n(F\Lotimes_A M) = \mrm{H}^0(F)\otimes_{\mrm{H}^0(A)} \mrm{H}^n(M).
\end{equation} 

Finally, we recall some other important tools, namely the DG-versions of the tensor and hom evaluation morphisms. There are natural morphisms
\begin{equation*}
	\tau\colon\RHom_A(L,M) \Lotimes_A N \to \RHom_A(L, M \Lotimes_A N)
\end{equation*}
and
\begin{equation*}\label{homev}
	\eta\colon L \Lotimes_A \RHom_A(M,N) \to \RHom_A(\RHom_A(L,M),N)
\end{equation*}
in $\msf{D}(A)$ called the \emph{tensor evaluation} and \emph{hom evaluation} respectively.
\begin{lem}\label{tensorhomeval}
	Let $A$ be a commutative noetherian DG-ring with bounded cohomology.  The tensor evaluation morphism $\tau$ is an isomorphism if any of the following conditions hold:
	\begin{itemize}
		\item[(i)] $L \in \msf{D}^\mrm{b}_\mrm{f}(A)$ and $\projdim_A(L) < \infty$;
		\item[(ii)] $N \in \msf{D}^\mrm{b}_\mrm{f}(A)$ and $\projdim_A(N) < \infty$;
		\item[(iii)] $L \in \msf{D}^-_\mrm{f}(A)$, $M \in \msf{D}^+(A)$ and $\flatdim_A(N) < \infty$;
		\item[(iv)] $\projdim_A(L) < \infty$, $M \in \msf{D}^-(A)$ and $N \in \msf{D}_\mrm{f}^-(A)$.
	\end{itemize}
	Furthermore, the hom evaluation morphism $\eta$ is an isomorphism provided either of the following conditions hold:
	\begin{itemize}
		\item[(i)] $L \in \msf{D}^-_\mrm{f}(A)$, $M \in \msf{D}^+(A)$ and $\injdim_A(N) < \infty$;
		\item[(ii)] $L \in \msf{D}^\mrm{b}_\mrm{f}(A)$ and $\projdim_A(L) < \infty$.
	\end{itemize}
\end{lem}
\begin{proof}
	We prove the isomorphisms for $\tau$, as the ones for $\eta$ are essentially identical in approach. For part (i), the assumptions imply that $L$ is compact in $\msf{D}(A)$ and then this is a standard property of compact objects. Part (ii) is similar. The assumptions in (iii) show that both $\RHom_A(-,M) \Lotimes_A N$ and $\RHom_A(-, M \Lotimes_A N)$ are contravariant way-out right functors. Since the map is an isomorphism for $L = A$ the result follows from the way-out lemma. Part (iv) is proved in a similar way to part (iii).
\end{proof}

\section{Amplitude bounds on derived homomorphisms}\label{sec:extbounds}
\subsection{Bounds on suprema and infima}
As projective and injective dimension of DG-modules can be determined by the vanishing of Ext functors, having bounds on the suprema and infima of particular DG-modules, and some associated complexes, is of great use. In this subsection, we prove some inequalities that will be of particular use throughout. 

We first recall from~\cite[Proposition 3.3]{ShINJ} and~\cite[Proof of Proposition 3.1]{YeDual} how reduction and coreduction interact with $\sup$ and $\inf$.

\begin{lem}\label{lem:infcoreduction}
	Let $A$ be a DG-ring.
	\begin{itemize}
		\item[(i)] If $M \in \msf{D}^+(A)$, then $\inf(\RHom_A(\mrm{H}^0(A),M)) = \inf(M)$ and moreover \[\mrm{H}^{\inf(M)}(\RHom_A(\mrm{H}^0(A),M)) = \mrm{H}^{\inf(M)}(M).\]
		\item[(ii)] If $M \in \msf{D}^-(A)$, then $\sup(\mrm{H}^0(A) \Lotimes_A M) = \sup(M)$ and moreover \[\mrm{H}^{\sup(M)}(\mrm{H}^0(A) \Lotimes_A M) = \mrm{H}^{\sup(M)}(M).\]
	\end{itemize}
\end{lem}

The above lemma can be used to obtain bounds on the cohomology of the derived tensor product and derived Hom functor.

\begin{lem}\label{lem:ExtLowerBound}
Let $A$ be a commutative DG-ring.
\begin{itemize}
\item[(i)] Suppose $M \in \cat{D}^{-}(A)$, $N \in \cat{D}^{+}(A)$, and let $s = \sup(M)$ and $i=\inf(N)$. Then
\[
\inf\left(\RHom_A(M,N)\right) \ge i-s.
\]
Moreover
\[
\mrm{H}^{i-s}\left(\RHom_A(M,N)\right) \cong
\opn{Hom}_{\mrm{H}^0(A)}(\mrm{H}^s(M),\mrm{H}^i(N)).
\]
\item[(ii)] Let $M, N \in \cat{D}^{-}(A)$. Then $\sup(M \Lotimes_A N) \leq \sup(M) + \sup(N)$. Moreover 
\[\mrm{H}^{\sup(M) + \sup(N)}(M \Lotimes_A N) = \mrm{H}^{\sup(M)}(M) \otimes_{\mrm{H}^0(A)} \mrm{H}^{\sup(N)}(N).\]
\end{itemize}
\end{lem}

\begin{proof}
We prove only the first item, as the second follows from a similar argument.
By adjunction, there is an isomorphism 
\[
\RHom_A(\mrm{H}^0(A),\RHom_A(M,N)) \cong
\RHom_{\mrm{H}^0(A)}(M\Lotimes_A \mrm{H}^0(A),\RHom_A(\mrm{H}^0(A),N))
\]
in $\msf{D}(\mrm{H}^0(A))$.
Therefore, we obtain that
\begin{align*}
	\inf\left(\RHom_A(M,N)\right) &= \inf(\RHom_A(\mrm{H}^0(A), \RHom_A(M,N))) &\mbox{ by \cref{lem:infcoreduction}}, \\ &= 
	\inf\left(
	\RHom_{\mrm{H}^0(A)}(M\Lotimes_A \mrm{H}^0(A),\RHom_A(\mrm{H}^0(A),N)) 
	\right) &\mbox{ by adjunction}, \\ &\ge \inf(\RHom_A(\mrm{H}^0(A),N))-\sup(M\Lotimes_A \mrm{H}^0(A)) &\mbox{ by~\cite[Lemma 2.1]{Foxby}}, \\ &= \inf(N)-\sup(M) &\mbox{ by \cref{lem:infcoreduction}}.
\end{align*}
This proves the first part of (i). For the second part, by applying \cref{lem:infcoreduction} and~\cite[Lemma 2.1]{Foxby} we obtain isomorphisms 
\begin{align*}
	\mrm{H}^{i-s}\left(\RHom_A(M,N)\right) &= 
	\mrm{H}^{i-s}\left(\RHom_{\mrm{H}^0(A)}(M\Lotimes_A \mrm{H}^0(A),\RHom_A(\mrm{H}^0(A),N))\right) \\ &=
	\opn{Hom}_{\mrm{H}^0(A)}(\mrm{H}^s(M\Lotimes_A \mrm{H}^0(A)),\mrm{H}^i(\RHom_A(\mrm{H}^0(A),N))) \\ &=
	\opn{Hom}_{\mrm{H}^0(A)}(\mrm{H}^s(M),\mrm{H}^i(N))
\end{align*}
as required.
\end{proof}
At certain stages, we will assume that our commutative DG-ring $A$ admits a dualizing DG-module. Recall that $R\in\msf{D}(A)$ is a dualizing DG-module provided the following properties hold:
\begin{itemize}
	\item[(i)] $R \in \msf{D}_\mrm{f}^+(A)$;
	\item[(ii)] the natural map $A \to \RHom_A(R,R)$ is an isomorphism;
	\item[(iii)] $\injdim_A(R) < \infty$.
\end{itemize}
    We note that the shift of any dualizing DG-module is also a dualizing DG-module, and moreover if $A$ is local then dualizing DG-modules over $A$ are unique up to shift. Therefore, we say that $R$ is normalized if $\inf(R) = -\mrm{dim}(\mrm{H}^0(A))$. We refer the reader
    to~\cite{FIJ,YeDual} for more details. 

The following proposition will be of use in subsequent sections.

\begin{prop}\label{prop:infExtDual}
	Let $A$ be a commutative noetherian DG-ring with bounded cohomology admitting a dualizing DG-module $R$. 
	Then for any $M \in \cat{D}^{+}(A)$ we have
	\[
	\inf\left(\RHom_A(R,M)\right) \le \inf(M)-\inf(R).
	\]
\end{prop}
\begin{proof}
	Write $i = \inf(M)$ and consider the $\mrm{H}^0(A)$-module $\mrm{H}^i(M)$, which is a non-zero module over the noetherian ring $\mrm{H}^0(A)$. Hence there is a prime ideal $\bar\p \in \opn{Spec}(\mrm{H}^0(A))$ such that $\bar\p \in \opn{Ass}(\mrm{H}^i(M))$. We note that $R_{\bar{\p}}$ is a dualizing DG-module over $A_{\bar{\p}}$, which, consequently, is non-zero; this implies that
	$\inf(R)\le \inf(R_{\bar{\p}}) \le \sup(R_{\bar{\p}})$.
	Using \cref{tensorhomeval} and adjunction, we have isomorphisms
	\[
	\RHom_A(R,M) \Lotimes_A A_{\bar{\p}} \cong 
	\RHom_A(R,M\Lotimes_A A_{\bar{\p}}) \cong
	\RHom_{A_{\bar{\p}}}(R_{\bar{\p}},M_{\bar{\p}}),
	\]
	and by our choice of $\bar{\p}$
	it follows that $i = \inf(M_{\bar{\p}})$.
	Let $s = \sup(R_{\bar{\p}})$.
	By \cref{lem:ExtLowerBound},
	we know
	\[
	\inf\left(\RHom_{A_{\bar{\p}}}(R_{\bar{\p}},M_{\bar{\p}})\right) \ge i-s
	\]
	and that
	\[
	\mrm{H}^{i-s}\left(\RHom_{A_{\bar{\p}}}(R_{\bar{\p}},M_{\bar{\p}})\right) \cong
	\opn{Hom}_{\mrm{H}^0(A_{\bar{\p}})}(\mrm{H}^s(R_{\bar{\p}}),\mrm{H}^i(M_{\bar{\p}})).
	\]
	By \cite[tag 0310]{SP},
	it holds that $\bar{\p}\mrm{H}^0(A_{\bar{\p}}) \in \opn{Ass}(\mrm{H}^i(M_{\bar{\p}}))$,
	while the fact that $R_{\bar{\p}}$ is a dualizing DG-module implies that $\mrm{H}^s(R_{\bar{\p}})$ is a non-zero finitely generated $\mrm{H}^0(A_{\bar{\p}})$-module.
	Since $\bar{\p}\mrm{H}^0(A_{\bar{\p}}) \in \opn{Ass}(\mrm{H}^i(M_{\bar{\p}}))$, there is an injective map \[\Hom_{\mrm{H}^0(A_{\bar{\p}})}(\mrm{H}^s(R_{\bar{\p}}),\mrm{H}^0(A_{\bar{\p}})/\bar{\p}\mrm{H}^0(A_{\bar{\p}}))\rightarrow \Hom_{\mrm{H}^0(A_{\bar{\p}})}(\mrm{H}^s(R_{\bar{\p}}),\mrm{H}^i(M_{\bar{\p}}))
	\]
	and the former is nonzero by Nakayama's lemma. Therefore
	\[
	\opn{Hom}_{\mrm{H}^0(A_{\bar{\p}})}(\mrm{H}^s(R_{\bar{\p}}),\mrm{H}^i(M_{\bar{\p}})) \ne 0,
	\]
	which shows that 
	$\inf\left(\RHom_{A_{\bar{\p}}}(R_{\bar{\p}},M_{\bar{\p}})\right) = i-s.$
	We therefore obtain that
	\[
	\inf\left(\RHom_A(R,M)\right) \le \inf\left(\RHom_{A_{\bar{\p}}}(R_{\bar{\p}},M_{\bar{\p}})\right) = i-s \le \inf(M) - \inf(R)
	\]
	as required.
\end{proof}
We note the preceding result, for complexes over rings, can be found at \cite[Proposition 2.2]{Foxby}. 
We obtain the following corollary.
\begin{cor}\label{cor:boundTns}
	Let $A$ be a commutative noetherian DG-ring with bounded cohomology,
	and let $R$ be a dualizing DG-module over $A$.
	Let $M\in \cat{D}^{\mrm{b}}(A)$ be a DG-module over $A$ of finite flat dimension.
	Then $\inf(R\Lotimes_A M) \ge \inf(R)+\inf(M)$.
\end{cor}
\begin{proof}
	Observe that
	\[
	M \cong \RHom_A(R,R)\Lotimes_A M \cong \RHom_A(R,R\Lotimes_A M)
	\]
	where the latter follows from~\cref{tensorhomeval}.
	It follows from \cref{lem:ExtLowerBound} that
	\[
	\inf M = \inf\left(\RHom_A(R,R\Lotimes_A M)\right) \le \inf(R\Lotimes_A M) - \inf(R),
	\]
	proving the claim.
\end{proof}

\subsection{Bounds for DG-modules of finite flat dimension}
We now turn our attention to the goal of proving the main theorem of this section, which provides a bound on the supremum of the derived homomorphisms between two DG-modules of finite flat dimension. Before stating and proving this theorem, we prove some intermediary results concerning injective dimension. To this end, we need a DG-version of the following fact: if $A$ is a noetherian ring, $F$ is a flat module and $E$ is an injective module, then $E\otimes_{A}F$ is also injective. This is obtained by the following formulation of Baer's criterion for DG-modules.

\begin{lem}\label{thm:injForm}
	Let $A$ be a DG-ring,
	and let $M \in \cat{D}^{\mrm{b}}(A)$.
	Then there is an equality
	\[
	\injdim_A(M) = \sup\{n\mid \opn{Ext}^n_A(\mrm{H}^0(A)/\bar{I},M) \ne 0, \mbox{ for some left ideal $\bar{I} \subseteq \mrm{H}^0(A)$}\}.
	\]
\end{lem}
\begin{proof}
	By \cref{cor:injdimcoreduction}
	there is an equality
	$
	\injdim_A(M) = \injdim_{\mrm{H}^0(A)}\left(\RHom_A(\mrm{H}^0(A),M)\right).
	$
	Combining this with \cite[Corollary 2.5.I]{AF},
	we have
	\[
	\injdim_{A}(M) = \sup\{n\mid \opn{Ext}^n_{\mrm{H}^0(A)}(\mrm{H}^0(A)/\bar{I},\mrm{RHom}_A(\mrm{H}^0(A),M)) \ne 0, \mbox{for some left ideal $\bar{I}\subseteq\mrm{H}^0(A)$}\}
	\]
	and therefore the result follows by adjunction.
\end{proof}
The following lemma yields the aforementioned DG-version of how injective dimension and flat dimension interact through tensor product.

\begin{prop}\label{prop:injdimFlatDG}
	Let $A$ be a commutative noetherian DG-ring,
	and let $J,F \in \cat{D}^{\mrm{b}}(A)$. 
	Assume that $J$ has finite injective dimension over $A$ and that $F$ has finite flat dimension over $A$.
	Then
	\[
	\injdim_A(J\Lotimes_A F) \le \injdim_A(J) + \sup(F).
	\]
\end{prop}

\begin{proof}
	Let $\bar{I} \subseteq \mrm{H}^0(A)$ be an ideal.
	The fact that $A$ is noetherian and $F$ has finite flat dimension over $A$ implies by \cref{tensorhomeval} that $\RHom_A(\mrm{H}^0(A)/\bar{I},J\Lotimes_A F) \cong \RHom_A(\mrm{H}^0(A)/\bar{I},J)\Lotimes_A F.$
	On the other hand, by definition we have that $\sup\left(\RHom_A(\mrm{H}^0(A)/\bar{I},J)\right) \le \injdim_A(J)$. Combining this with \cref{lem:ExtLowerBound}(ii) we have isomorphisms
	\begin{align*}
		\sup\left(\RHom_A(\mrm{H}^0(A)/\bar{I},J\Lotimes_A F)\right) &=
		\sup\left(\RHom_A(\mrm{H}^0(A)/\bar{I},J)\Lotimes_A F\right) \\ &\le
		\sup\left(\RHom_A(\mrm{H}^0(A)/\bar{I},J)\right) + \sup(F) \\ &\le \injdim_A(J) + \sup(F).
	\end{align*}
	Therefore, for all ideals $\bar{I} \subseteq \mrm{H}^0(A)$ and for all $n>\injdim_A(J) + \sup(F)$ we see 
	\[
	\opn{Ext}^n_A(\mrm{H}^0(A)/\bar{I},J\Lotimes_A F) = 0.
	\]
	Applying \cref{thm:injForm} to this vanishing implies that $
	\injdim_A(J\Lotimes_A F) \le \injdim_A(J) + \sup(F)$
	as required.
\end{proof}

Over a commutative noetherian local ring $A$, it is particularly helpful to know the injective dimension of a dualizing complex $R$, should it exist. In this case, it is well known (see, for example~\cite[tag 0BUJ]{SP}) that $\injdim_{A}(R)=\inf(R)+\text{dim}(A)$. The next result provides an analogy to this result in the DG-setting.

\begin{lem}\label{prop:dualInjDG}
	Let $(A,\bar{\m})$ be a noetherian local DG-ring with bounded cohomology,
	and let $R$ be a dualizing DG-module over $A$.
	Then 
	\[
	\injdim_A(R) = \inf(R) + \dim(\mrm{H}^0(A)).
	\]
\end{lem}
\begin{proof}
	The fact that $A$ has bounded cohomology implies by \cite[Corollary 7.3]{YeDual} that $R$ is bounded.
	By \cref{cor:injdimcoreduction} we know that
	\[
	\injdim_A(R) = \injdim_{\mrm{H}^0(A)}\left(\RHom_A(\mrm{H}^0(A),R)\right).
	\]
	By \cite[Proposition 7.5]{YeDual},
	$\RHom_A(\mrm{H}^0(A),R)$ is a dualizing complex over $\mrm{H}^0(A)$ and hence, combining the above with the aforementioned result for rings, we obtain that
	\[
	\injdim_A(R) = \injdim_{\mrm{H}^0(A)}(\RHom_A(\mrm{H}^0(A),R)) = \inf\left(\RHom_A(\mrm{H}^0(A),R)\right) + \dim(\mrm{H}^0(A)).
	\]
	Finally, note that \[
	\inf\left(\RHom_A(\mrm{H}^0(A),R)\right) = \inf(R),
	\]
	by \cref{lem:infcoreduction}, hence proving the result.
\end{proof}
We may now give the main theorem of this section, which will be of significant use in what follows. This extends~\cite[Theorem 2.6]{Foxby} to the DG-setting.
\begin{thm}\label{thm:ExtBoundFlat}
	Let $(A,\bar{\m})$ be a noetherian local DG-ring with bounded cohomology,
	and suppose $A$ has a dualizing DG-module $R$.
	Let $M,N \in \cat{D}^{\mrm{b}}(A)$ be two DG-modules of finite flat dimension over $A$. 
	Then 
	\[
	\sup\left(\RHom_A(M,N)\right) \le \sup(N)-\inf(M) + \dim(\mrm{H}^0(A)).
	\]
\end{thm}
\begin{proof}
	By definition of dualizing DG-module, \cref{tensorhomeval}, and adjunction, we have isomorphisms
	\begin{align*}
		\RHom_A(M,N) &\cong \RHom_A(M,N\Lotimes_A \RHom_A(R,R)) \\ 
		&\cong \RHom_A(M,\RHom_A(R,N\Lotimes_A R)) \\ 
		&\cong \RHom_A(M\Lotimes_A R,N\Lotimes_A R)
	\end{align*}
	in $\msf{D}(A)$. 
	By \cref{cor:boundTns}, 
	we know that $\inf(M\Lotimes_A R) \ge \inf(R)+\inf(M)$.
	
	By applying \cref{prop:injdimFlatDG} and \cref{prop:dualInjDG} in turn it follows that
	\[
	\injdim_A(N\Lotimes_A R) \le \injdim_A(R) + \sup(N) = \sup(N) + \inf(R) + \dim(\mrm{H}^0(A)).
	\]
	Combining these observations, it follows that
	\begin{align*}
		\sup\left(\RHom_A(M,N)\right) &= \sup\left(\RHom_A(M\Lotimes_A R,N\Lotimes_A R)\right) \\ 
		&\le
		\injdim_A(N\Lotimes_A R) - \inf(M\Lotimes_A R) \\ 
		&\le \sup(N) + \inf(R) + \dim(\mrm{H}^0(A)) -\inf(R) -\inf(M) \\
		&=\sup(N)-\inf(M) + \dim(\mrm{H}^0(A))
	\end{align*}
	where the first inequality holds by definition of injective dimension.
\end{proof}

\section{Bounding projective dimension over local DG-rings}\label{sec:local}
In this section we prove a bound on the projective dimension of DG-modules with finite flat dimension under the assumption that the DG-ring is local. More precisely, we will prove the following.
\begin{thm}\label{localprojdim}
Let $(A,\bar{\m})$ be a commutative noetherian local DG-ring with bounded cohomology.
Let $M \in \cat{D}^\mrm{b}(A)$, and suppose that $\flatdim_A(M) < \infty$.
Then $\projdim_A(M) \le  \dim(\mrm{H}^0(A)) -\inf(M)$.
\end{thm}
In the next section we will globalize this result, see \cref{thm:globalfpdbound}.

A key tool in abelian categories is dimension shifting. The following lemma illustrates an analogous construction exists in $\msf{D}(A)$, and has many of the desired properties. It will be of fundamental use in the subsequent proof.

\begin{prop}\label{prop:sppj}
	Let $A$ be a DG-ring, 
	and let $M \in \cat{D}^{-}(A)$ be a non-zero DG-module.
	\begin{enumerate}
			\item\label{sppjexists} There exists a map $f\colon P\to M[\sup(M)]$ in $\cat{D}(A)$ such that $P \in \mrm{Proj}(A)$ and $\mrm{H}^0(f)$ is surjective.
			\item\label{extcalc} If we embed such a map $f$ in a distinguished triangle
			\[
			N \to P \xrightarrow{f} M[\sup(M)] \to N[1]
			\]
			in $\cat{D}(A)$,
			then for any $L \in \cat{D}^{-}(A)$ and any $n> \sup(L)$ there are natural isomorphisms
			\[
			\opn{Ext}^n_A(N,L) \cong \opn{Ext}^{n+1}_A(M[\sup(M)],L) \cong \opn{Ext}^{n+1-\sup(M)}_A(M,L).
			\]
		\end{enumerate}
\end{prop}
\begin{proof}
	For any projective $\mrm{H}^0(A)$-module $\bar{P}$, the functor $\Hom_{\mrm{H}^0(A)}(\bar{P}, \mrm{H}^0(-))\colon \msf{D}(A) \to \mrm{Mod}(\mrm{H}^0(A))$ is homological and product-preserving, so part (1) follows by applying Brown representability. For part (2), applying the functor $\RHom_A(-,L)$ and taking the long exact sequence in cohomology we obtain an exact sequence
	\[\cdots \to \mrm{Ext}_A^n(P,L) \to \mrm{Ext}_A^n(N,L) \to \mrm{Ext}_A^{n+1}(M[\sup(M)], L) \to \mrm{Ext}_A^{n+1}(P,L) \to \cdots.\]
	Since $P \in \mrm{Proj}(A)$, the outer terms vanish for $n>\sup(L)$ proving (2).
\end{proof}

We are now in a position to turn our attention to the proof of \cref{localprojdim}. The subsequent proposition provides the main body of the proof.

\begin{prop}\label{thm:upperBoundamp}
Let $(A,\bar{\m})$ be a commutative noetherian local DG-ring with bounded cohomology which has a dualizing DG-module.
Let $M \in \cat{D}^\mrm{b}(A)$ with $\amp(M) \geq \amp(A)$, and suppose that $\flatdim_A(M) < \infty$.
Then $\projdim_A(M) \le  \dim(\mrm{H}^0(A)) -\inf(M)$.
\end{prop}
\begin{proof}
By replacing $M$ by $M[\sup(M)]$ and using \cref{lem:pdshift}, it suffices to prove the claim when $\sup(M) = 0$. In this case, we must show that $\projdim_A(M) \le  \dim(\mrm{H}^0(A)) + \amp(M)$. 
To simplify notation we write $N = \amp(M)$. 

We construct an sppj resolution of $M$, 
in the sense of \cite[Definition 2.17]{Mi}.
By \cref{prop:sppj}(\ref{sppjexists}) we may choose a map $f_0\colon P_0 \to M$ in $\cat{D}(A)$ such that $P_0 \in \mrm{Proj}(A)$ and $\mrm{H}^0(f_0)\colon \mrm{H}^0(P_0) \to \mrm{H}^0(M)$ is surjective.
As $P_0 \in \opn{Proj}(A)$, it follows that $\sup(P_0) = 0$,
and that $\amp(P_0) \le \amp(A)$ as $\mrm{H}^i(P_0) = \mrm{H}^i(P_0 \Lotimes_A A) = \mrm{H}^0(P_0) \otimes_{\mrm{H}^0(A)} \mrm{H}^i(A)$ by \cref{eqn:flat}.
We may extend the map $f_0$ to a distinguished triangle
\[
M_0 \xrightarrow{g_0} P_0 \xrightarrow{f_0} M \to M_0[1].
\]
As both $P_0$ and $M$ have finite flat dimension,
it follows that $M_0$ has finite flat dimension.

For any $i$, we have an exact sequence
\begin{equation}\label{eq:cohex}
\ldots \to \mrm{H}^{-(i+1)}(M) \to \mrm{H}^{-i}(M_0) \to \mrm{H}^{-i}(P_0) \to \ldots
\end{equation}
The fact that $\sup(P_0) = \sup(M) = 0$ and that $\mrm{H}^0(f_0)$ is surjective implies that $\sup(M_0) \le 0$.

If $i > N$, we claim that $\mrm{H}^{-i}(M_0) = 0$. Firstly, $\mrm{H}^{-i}(P_0) = 0$ since $i>N \ge \amp(A) \ge -\inf(P_0)$. On the other hand, since $i > N \ge \amp(M)$, we also deduce that $\mrm{H}^{-(i+1)}(M) = 0$ so the claim follows from the exact sequence~(\ref{eq:cohex}).
It follows that $\inf(M_0) \ge -N$.
Since it may now happen that $\sup(M_0)<0$,
we will also consider $M_0[\sup(M_0)]$.
Note that $\inf(M_0[\sup(M_0)]) = \inf(M_0) - \sup(M_0) \ge \inf(M_0) \ge -N$.

We now construct the next step of the resolution,
but observe that unlike the previous step, it might happen now that $\sup(M_0)$ is strictly less than $0$.
Thus, we now take an element $P_{-1} \in \opn{Proj}(A)$ and a map $f_{-1}:P_{-1} \to M_0[\sup(M_0)]$
such that $\mrm{H}^0(f_{-1})$ is surjective.
Then, we embed $f_{-1}$ into a distinguished triangle of the form
\[
M_{-1} \xrightarrow{g_{-1}} P_{-1} \xrightarrow{f_{-1}} M_0[\sup(M_0)] \to M_{-1}[1].
\]
We have seen above that $M_0[\sup(M_0)]$ satisfies the exact same hypothesis as $M$,
so all of the above conclusions remain valid when $M$ is replaced by $M_0[\sup(M_0)]$ and $M_0$ is replaced by $M_{-1}$. We now continue in the same fashion, until we arrive to the distinguished triangle
\[
M_{-n} \xrightarrow{g_{-n}} P_{-n} \xrightarrow{f_{-n}} M_{-(n-1)}[\sup(M_{-(n-1)})] \to M_{-n}[1].
\]
where $n = N + \dim(\mrm{H}^0(A)) + 1$.

For ease of notation we set the convention that $M_1 = M$. Let $k \leq n$, and consider the defining triangle
\[M_{-(k-i)} \to P_{-(k-i)} \to M_{-(k-i-1)}[\sup(M_{-(k-i-1)})] \to M_{-(k-i)}[1]\]
where $1 \leq i \leq k$. By applying \cref{prop:sppj}(\ref{extcalc}), for each $1 \leq i \leq k$ in this range we obtain isomorphisms
\begin{align*}
\mrm{Ext}^{i-\sum_{j=1}^{i}\sup(M_{-(k-j)})}_A(M_{-(k-i)}, M_{-k}) &\cong \mrm{Ext}^{i +1- \sup(M_{-(k-i-1)}) - \sum_{j=1}^{i}\sup(M_{-(k-j)})}_A(M_{-(k-i-1)}, M_{-k})  \\
&\cong \mrm{Ext}^{i+1-\sum_{j=1}^{i+1}\sup(M_{-(k-j)})}_A(M_{-(k-i-1)}, M_{-k})
\end{align*}
as $i-\sum_{j=1}^{i}\sup(M_{-(k-j)}) \geq i > 0 \geq \sup(M_{-k}).$
By repeatedly applying these isomorphisms we obtain a chain of isomorphisms
\begin{align*}
\opn{Ext}^1_A(M_{-(k-1)}[\sup(M_{-(k-1)})],M_{-k}) &\cong \opn{Ext}^{1-\sup(M_{-(k-1)})}_A(M_{-(k-1)},M_{-k}) \\
&\cong \opn{Ext}^{2-\sup(M_{-(k-1)})-\sup(M_{-(k-2)})}_A(M_{-(k-2)},M_{-k}) \\
&\cong \cdots \\
&\cong \opn{Ext}^{1+k-(\sum_{j=1}^{k+1} \sup(M_{-(k-j)}))}_A(M,M_{-k}) \\
&\cong \opn{Ext}^{1+k-(\sum_{j=0}^{k-1} \sup(M_{-j}))}_A(M,M_{-k}).
\end{align*}

Consider the sequence $a_i = 1 + i-(\sum_{j=0}^i \sup(M_{-j}))$.
This is a strictly increasing sequence of natural numbers,
and moreover, $a_n > n =  N + \dim(\mrm{H}^0(A)) + 1 > N + \dim(\mrm{H}^0(A))$.
Let $k$ be the smallest integer such that $a_k \ge n$.

Since $M$ (and hence $M_{-k}$) have finite flat dimension, by \cref{thm:ExtBoundFlat} it holds that
\begin{equation}\label{eqn:finalRHomSup}
\sup\left(\RHom_A(M,M_{-k})\right) \le \sup(M_{-k}) -\inf(M) +\dim(\mrm{H}^0(A)) = \sup(M_{-k}) + n-1.
\end{equation}

Now we have
\begin{align*}
1+k-\sum_{j=0}^{k-1} \sup(M_{-j}) &= a_k + \sup(M_{-k}) \\[-10pt]
&\ge n + \sup(M_{-k}) \\
&> \sup(M_{-k}) + n-1 \\
&\ge \sup\left(\RHom_A(M,M_{-k})\right)
\end{align*}
by  (\ref{eqn:finalRHomSup}) and hence it follows that
\[
\opn{Ext}^1_A(M_{-(k-1)}[\sup(M_{-(k-1)})],M_{-k}) \cong
\opn{Ext}^{1+k-(\sum_{j=0}^{k-1} \sup(M_{-j}))}_A(M,M_{-k}) \cong 0.
\]

Applying \cref{prop:dgSplit} to the distinguished triangle
\[
M_{-k} \xrightarrow{g_{-k}} P_{-k} \xrightarrow{f_{-k}} M_{-(k-1)}[\sup(M_{-(k-1)})] \to M_{-k}[1].
\]
this in turn implies that $M_{-(k-1)}[\sup(M_{-(k-1)})] \in \opn{Proj}(A)$.

In order to deduce that $\projdim_A(M) \leq n-1 = N + \mrm{dim}(\mrm{H}^0(A))$ it remains to show that for any $K \in \msf{D}^0(A)$ with $\inf(K) = \sup(K) = 0$, we have $\mrm{Ext}^n_A(M,K) = 0$. 
By repeatedly applying \cref{prop:sppj}(\ref{extcalc}) we obtain isomorphisms
\begin{align*}
\opn{Ext}^n_A(M,K) &\cong \opn{Ext}^{n-1}_A(M_0,K) \\
&\cong \opn{Ext}^{n-2+\sup(M_0)}_A(M_{-1},K) \\
&\cong \dots \\
&\cong \opn{Ext}^{n-(k-\sum_{j=0}^{k-2} \sup(M_{-j}))}_A(M_{-(k-1)},K) \\
&\cong \opn{Ext}^{n-(k-\sum_{j=0}^{k-1}\sup(M_{-j}))}_A(M_{-(k-1)}[\sup(M_{-(k-1)})],K) \\
&= \opn{Ext}^{n-a_{k-1}}_A(M_{-(k-1)}[\sup(M_{-(k-1)})],K).
\end{align*}
Since $k$ was defined to be the smallest choice for which $a_k \geq n$, it follows that $a_{k-1} < n$. Hence as $M_{-(k-1)}[\sup(M_{-(k-1)})] \in \opn{Proj}(A)$, the above $\opn{Ext}$-module is $0$,
proving the claim.
\end{proof}

Having proved the case when $\amp(M) \geq \amp(A)$, we may deduce the bound also for $M$ with $\amp(M) < \amp(A)$ as follows.
\begin{prop}\label{thm:upperBound}
Let $(A,\bar{\m})$ be a commutative noetherian local DG-ring with bounded cohomology which has a dualizing DG-module.
Let $M \in \cat{D}^\mrm{b}(A)$, and suppose that $\flatdim_A(M) < \infty$.
Then $\projdim_A(M) \le  \dim(\mrm{H}^0(A)) -\inf(M)$.
\end{prop}
\begin{proof}
We may assume that $\amp(M) < \amp(A)$;
otherwise the result follows directly from \cref{thm:upperBoundamp}.
By shifting if necessary, 
we may assume that $\sup(M) = 0$, 
and we must show that $\projdim_A(M) \le  \dim(\mrm{H}^0(A)) + \amp(M)$.
Let $n = \amp(A) - \amp(M) > 0$,
and let $K = M \oplus M[n]$.
Note that $\inf(K) = \inf(M)-n = \inf(A)$,
and that $\sup(K) = 0$, so that $\amp(K) = \amp(A)$.
Moreover, it is clear that $\flatdim_A(K) < \infty$.
Hence, by \cref{thm:upperBoundamp} it follows that $\projdim_A(K) \le \dim(\mrm{H}^0(A)) + \amp(A)$.
On the other hand, by \cref{lem:pdshift} we have that
\[
\projdim_A(K) = \max(\projdim_A(M),\projdim_A(M) + n) = \projdim_A(M) + n.
\]
Thus, we deduce that
\[
\projdim_A(M) \le \dim(\mrm{H}^0(A)) + \amp(A) - n = \dim(\mrm{H}^0(A)) + \amp(M),
\]
as claimed.
\end{proof}

Having treated the case when $A$ has a dualizing DG-module, we now turn to deducing the local case in full generality. In order to so, we need to state a couple of elementary properties of faithfully flat descent.
\begin{dfn}
A map $f\colon A \to B$ of commutative DG-rings is \emph{faithfully flat} if $B\in \mrm{Flat}(A)$ (equivalently, if $\flatdim_A(B) = 0$) and the induced map $\mrm{H}^0(f)\colon \mrm{H}^0(A) \to \mrm{H}^0(B)$ is faithfully flat.
\end{dfn}

\begin{lem}\label{prop:ffamp}
	Let $f\colon A \to B$ be a faithfully flat map of commutative DG-rings. Then for any $M \in \msf{D}(A)$, we have $\sup(B \Lotimes_A M) = \sup(M)$, $\inf(B \Lotimes_A M) = \inf(M)$ and $\amp(B \Lotimes_A M) = \amp(M)$. 
\end{lem}
\begin{proof}
	By \cref{eqn:flat} we have $\mrm{H}^n(B \Lotimes_A M) = \mrm{H}^0(B) \otimes_{\mrm{H}^0(A)} \mrm{H}^n(M).$ Since $\mrm{H}^0(A) \to \mrm{H}^0(B)$ is faithfully flat, $\mrm{H}^n(M) = 0$ if and only if $\mrm{H}^0(B) \otimes_{\mrm{H}^0(A)} \mrm{H}^n(M) = 0$ and the result follows.
\end{proof}

Recall that for a faithfully flat map of commutative rings $A \to B$,
we have $\projdim_B(M\Lotimes_A B) = \projdim_A(M)$ for any $M \in \msf{D}(A)$ of finite projective dimension, see for example the remark following~\cite[Theorem 14.1.21]{dcmca}. Combining this observation with \cref{cor:injdimcoreduction} and \cref{prop:ffamp} one obtains the following lemma.
\begin{lem}\label{prop:projDimFFDG}
Let $A \to B$ be a faithfully flat map of commutative DG-rings,
and let $M \in \cat{D}(A)$ be such that $\projdim_A(M) < \infty$.
Then $\projdim_B(M\Lotimes_A B) = \projdim_A(M)$.
\end{lem}

Accumulating the previous results, we are now in a position to prove \cref{localprojdim}.
\begin{proof}[Proof of \cref{localprojdim}]
Let us write $\widehat{A}:= \msf{L}\Lambda(A,\bar{\m})$
for the derived $\bar{\m}$-completion of $A$, see~\cite{Shaul2019} for more details.
The isomorphism
\[
(M \Lotimes_A \widehat{A}) \Lotimes_{\widehat{A}} - \cong M\Lotimes_A -
\]
shows that $\flatdim_{\widehat{A}}(M \Lotimes_A \widehat{A}) < \infty$.
By \cite[Proposition 7.21]{ShINJ},
the DG-ring $\widehat{A}$ has a dualizing DG-module,
and hence $\projdim_{\widehat{A}}(M \Lotimes_A \widehat{A}) \le  \dim(\mrm{H}^0(\widehat{A})) -\inf(M \Lotimes_A \widehat{A})$ by \cref{thm:upperBound}.

Since $A \to \widehat{A}$ is faithfully flat by~\cite[Corollary 4.6]{ShHomDim} and~\cite[Proposition 4.16]{Shaul2019},
by applying~\cref{prop:ffamp} we have
$\inf(M) = \inf(M \Lotimes_A \widehat{A})$.
Moreover, as $\mrm{H}^0(\widehat{A}) = \widehat{\mrm{H}^0(A)}$ by~\cite[Proposition 4.16]{Shaul2019},
we know that $\dim(\mrm{H}^0(\widehat{A})) = \dim(\mrm{H}^0(A))$.
Since
\[
\projdim_A(M) = \projdim_{\widehat{A}}(M \Lotimes_{A} \widehat{A}) 
\]
by \cref{prop:projDimFFDG}, the result follows.
\end{proof}

\section{Local-to-global for projective dimension}\label{sec:global}
The aim of this section is to remove the local assumption from \cref{localprojdim}.
There is an obstacle to performing a typical local-to-global argument, namely that, in general, projective dimension cannot be checked at localization at prime ideals. For example, there are flat modules $P$ over commutative noetherian rings $A$ of finite Krull dimension,
such that $P$ is not projective over $A$, but $P_{\p}$ is a projective $A_{\p}$-module for all $\p \in \opn{Spec}(A)$. Such examples exist even over rings as simple as $\mathbb{Z}$, see \cite[Exercise 19.12]{Eisenbud} for an example.

To overcome this obstacle we turn to semi-flat cotorsion replacements over noetherian rings of finite Krull dimension, as developed in~\cite{NT}. Recall that, over any ring $A$, a module $M$ is \textit{cotorsion} if $\text{Ext}_{A}^{1}(F,M)=0$ for every flat $A$-module $F$. Over commutative noetherian rings, the modules which are simultaneously flat and cotorsion, often called flat cotorsion modules, were completely described in~\cite{Enochs}: they are precisely the modules of the form $\prod_{\p \in \opn{Spec}(A)} T_{\p}$ where each $T_{\p}$ is the $\p$-adic completion of a free $A_{\p}$-module.

The next result is extracted from \cite{NT}.

\begin{lem}\label{prop:semiflat}
Let $A$ be a commutative noetherian ring with $d = \dim(A) < \infty$,
and let $N$ be an $A$-module.
For each $0 \le i \le d$ 
	let
	\[
	W_i = \{\p \in \opn{Spec}(A) \mid \dim(A/\p) = i\}.
	\]
	Then there exists a complex $Y \in \cat{D}(A)$ which satisfies the following properties:
	\begin{enumerate}
		\item There is an isomorphism $Y \cong N$ in $\cat{D}(A)$.
		\item There is a filtration of subcomplexes of $A$-modules
		\[
		0 = Y_{d+1} \subsetneq Y_d \subsetneq Y_{d-1} \subsetneq \dots \subsetneq Y_1 \subsetneq Y_0 = Y.
		\]
		\item For each $j$, there is a direct product decomposition
		$Y_{j-1}/Y_j \cong \prod_{\p \in W_{j-1}} T(\p)$, where $T(\p)$ is a complex of flat cotorsion $A_{\p}$-modules with cosupport contained in $\{\p\}$. 
		Moreover, it satisfies the property that $\sup(T(\p)) \le \dim(A/\p)$.
	\end{enumerate}
\end{lem}
\begin{proof}
	Let $F \cong N$ be a flat resolution,
	and let $Y = A^{\mathbb{W}}F$,
	the minimal semi-flat cotorsion replacement of $F$ over $A$,
	described in \cite[Construction 3.3]{NT} and \cite[Theorem 3.4]{NT}.
	Then by definition $Y \cong F \cong N$,
	and as explained in the proof of \cite[Theorem 4.9]{NT},
	\cite[(1.17)]{NT} gives the required filtration.
\end{proof}

We now turn to the first main theorem in this paper, which gives the global case of \cref{localprojdim}. This generalizes the results of Raynaud-Gruson~\cite[Theorem II.3.2.6]{RG} to the setting of DG-rings.
\begin{thm}\label{thm:globalfpdbound}
	Let $A$ be a commutative noetherian DG-ring with bounded cohomology.
	Let $M \in \cat{D}^{\mrm{b}}(A)$, and suppose that $\flatdim_A(M) < \infty$.
	Then $\projdim_A(M) \le  \dim(\mrm{H}^0(A)) -\inf(M)$.
\end{thm}
\begin{proof}
	There is nothing to prove if $\dim(\mrm{H}^0(A)) = \infty$,
	so assume that $\dim(\mrm{H}^0(A)) = d < \infty$.
	To show that $\projdim_A(M) \le  \dim(\mrm{H}^0(A)) -\inf(M)$ we use
	\cref{prop:zeroAmp}. 
	Thus, for any $N \in \cat{D}^0(A)$ we must show that $\opn{Ext}^i_A(M,N) = 0$ for all $i>\dim(\mrm{H}^0(A)) -\inf(M)$.
	By shifting if necessary, we can without loss of generality assume that $\sup(N) = \inf(N) = 0$.
	As $N \cong \mrm{H}^0(N)$, and as the latter is a $\mrm{H}^0(A)$-module,
	we may consider $N$ as a $\mrm{H}^0(A)$-module.
	We now proceed as in the proof of \cite[Theorem 4.9]{NT},
	and use \cref{prop:semiflat} to first replace $N$ by a flat resolution $F \to N$ in $\cat{D}(\mrm{H}^0(A))$,
	and then replace $F$ by $Y = A^{\mathbb{W}}F$,
	the minimal semi-flat cotorsion replacement of $N$ over $\mrm{H}^0(A)$.
	Here, we have set
	\[
	W_i = \{\bar{\p} \in \opn{Spec}(\mrm{H}^0(A)) \mid \dim \mrm{H}^0(A)/\bar{\p} = i\},
	\]
	In particular, there is an isomorphism $N \cong Y$ in $\cat{D}(\mrm{H}^0(A))$,
	and hence also in $\cat{D}(A)$,
	so $\opn{Ext}^i_A(M,N) \cong \opn{Ext}^i_A(M,Y)$.
	It is thus enough to show that 
	$\opn{Ext}^i_A(M,Y) = 0$ for all $i>\dim(\mrm{H}^0(A)) -\inf(M)$.
	As explained above, there is a filtration of subcomplexes of $\mrm{H}^0(A)$-modules:
	\[
	0 = Y_{d+1} \subsetneq Y_d \subsetneq Y_{d-1} \subsetneq \dots \subsetneq Y_1 \subsetneq Y_0 = Y.
	\]
	Take some $i>\dim(\mrm{H}^0(A)) -\inf(M)$.
	To show that $\opn{Ext}^i_A(M,Y) = 0$ we proceed by descending induction,
	and show that for all $j$ it holds that $\opn{Ext}^i_A(M,Y_j) = 0$.
	For $j=d+1$, $Y_j = 0$ so there is nothing to show.
	Assuming we have shown that $\opn{Ext}^i_A(M,Y_j) = 0$,
	consider the short exact sequence of complexes of $\mrm{H}^0(A)$-modules:
	\[
	0 \to Y_j \to Y_{j-1} \to Y_{j-1}/Y_j \to 0
	\]
	This short exact sequence of complexes give rise to a distinguished triangle
	\[
	Y_j \to Y_{j-1} \to Y_{j-1}/Y_j \to Y_j[1]
	\]
	in $\cat{D}(\mrm{H}^0(A))$, so applying the forgetful functor,
	we also obtain the same distinguished triangle in $\cat{D}(A)$.
	Applying the functor $\RHom_A(M,-)$ to it,
	we obtain a distinguished triangle
	\[
	\RHom_A(M,Y_j) \to \RHom_A(M,Y_{j-1}) \to \RHom_A(M,Y_{j-1}/Y_j) \to \RHom_A(M,Y_j)[1]
	\]
	Hence, there is an exact sequence of $\mrm{H}^0(A)$-modules of the form
	\[
	\opn{Ext}^i_A(M,Y_j) \to \opn{Ext}^i_A(M,Y_{j-1}) \to \opn{Ext}^i_A(M,Y_{j-1}/Y_j)
	\]
	By assumption, the leftmost term is zero,
	so in order for us to show that $\opn{Ext}^i_A(M,Y_{j-1}) = 0$,
	it is thus enough to show that $\opn{Ext}^i_A(M,Y_{j-1}/Y_j) = 0$.
	This was also concluded in \cite[Theorem 4.9]{NT},
	and the above shows this remains true in this setting.
	As stated in \cref{prop:semiflat},
	$Y_{j-1}/Y_j$ is a complex of $\mrm{H}^0(A)$-modules which decomposes as a direct product 
	$\prod_{\bar{\p} \in W_{j-1}} T(\bar{\p})$, where $T(\bar{\p})$ is a complex of flat cotorsion $\mrm{H}^0(A)$-modules with cosupport in $\{\bar{\p}\}$. 
	Because of this direct product decomposition, it is thus enough to show that
	$\opn{Ext}^i_A(M,T(\bar{\p})) = 0$.
	Since $T_{\bar{\p}}$ is a complex of $\mrm{H}^0(A)_{\bar{\p}}$-modules,
	we may consider it as an object of $\cat{D}(A_{\bar{\p}})$,
	which implies by adjunction that
	\[
	\opn{Ext}^i_A(M,T(\bar{\p})) \cong \opn{Ext}^i_{A_{\bar{\p}}}(M_{\bar{\p}},T(\bar{\p})).
	\]
	Note that $\flatdim_{A_{\bar{\p}}}(M_{\bar{\p}}) < \infty$,
	so by \cref{localprojdim} we know that
	\[
	\projdim_{A_{\bar{\p}}}(M_{\bar{\p}}) \le \dim(\mrm{H}^0(A_{\bar{\p}})) - \inf(M_{\bar{\p}}) \le \dim(\mrm{H}^0(A_{\bar{\p}})) - \inf(M).
	\]
	This means by definition that
	$\opn{Ext}^i_{A_{\bar{\p}}}(M_{\bar{\p}},T(\bar{\p})) = 0$
	provided that 
	\[
	i> \dim(\mrm{H}^0(A_{\bar{\p}})) - \inf(M) + \sup(T(\bar{\p})).
	\]
	By \cref{prop:semiflat},
	\begin{align*}
		\dim(\mrm{H}^0(A_{\bar{\p}})) - \inf(M) + \sup(T(\bar{\p})) &\le \dim(\mrm{H}^0(A_{\bar{\p}})) + \dim(\mrm{H}^0(A)/\bar{\p}) - \inf(M) \\ &\le \dim(\mrm{H}^0(A)) - \inf(M).
	\end{align*}
	Hence, if $i>\dim(\mrm{H}^0(A)) - \inf(M)$ then in particular
	$i > \dim(\mrm{H}^0(A_{\bar{\p}})) - \inf(M) + \sup(T(\bar{\p}))$
	which shows that $\opn{Ext}^i_A(M,T(\bar{\p})) = 0$,
	as claimed.
\end{proof}

\begin{rem}\label{rem:reduction-bound}
	Since by \cref{cor:injdimcoreduction} we know it is possible to reduce the computation of projective dimension over $A$ to a computation of projective dimension over $\mrm{H}^0(A)$,
	the reader might wonder whether it is possible to use this reduction to deduce \cref{thm:globalfpdbound} directly from the corresponding result for complexes over rings.
	However, this is not possible in general. Even though the analogous result of \cref{thm:globalfpdbound} holds for complexes over rings (see for example~\cite[Corollary 8.5.18]{dcmca}),
	the only conclusion one can obtain is that if $M \in \cat{D}^{\mrm{b}}(A)$ has $\flatdim_A(M) < \infty$,
	then
	\[
	\projdim_A(M) = \projdim_{\mrm{H}^0(A)}(M\Lotimes_A \mrm{H}^0(A)) \le \dim(\mrm{H}^0(A)) - \inf(M\Lotimes_A \mrm{H}^0(A)).
	\]
	However, we typically have no way to bound $\inf(M\Lotimes_A \mrm{H}^0(A))$.
	The best one can say in general is that
	\[
	-\inf(M\Lotimes_A \mrm{H}^0(A)) \le \flatdim_A(M),
	\]
	and hence
	\[
	\projdim_A(M) \le \dim(\mrm{H}^0(A)) + \flatdim_A(M),
	\]
	which does not provide sufficient information.
	There is one special case where one can explicitly bound $\inf(M\Lotimes_A \mrm{H}^0(A))$ for any $M$ of finite flat dimension, namely when $A$ and $\mrm{H}^0(A)$ are Gorenstein local.
	In that case, this reduction method yields a sharper bound on $\projdim_A(M)$;
	see \cref{thm:GorensteinFPD} below.
\end{rem}

\section{Constructing DG-modules of prescribed projective dimension}\label{sec:construction}
The goal of this section is to prove \cref{thm:construction}, which, given any commutative noetherian DG-ring $A$ with bounded cohomology, constructs a DG-module of projective dimension $\dim(\mrm{H}^0(A))$ and amplitude at most $\amp(A)$.
This generalizes the construction of Bass in \cite{Bass}. We will use this to give an explicit lower bound on the big finitistic dimension in \cref{cor:FPDbounds}.

For a commutative noetherian local DG-ring $(A, \bar{\m})$ with bounded cohomology, we write $\msf{R}\Gamma_{\bar{\m}}$ for the (derived) local cohomology functor, see~\cite{Shaul2019} for more details. We recall that a commutative noetherian local DG-ring $(A, \bar{\m})$ with bounded cohomology is said to be \emph{local-Cohen-Macaulay} if $\amp(\msf{R}\Gamma_{\bar\m}(A)) = \amp(A),$ and is said to be \emph{Cohen-Macaulay} if $A_{\bar{\p}}$ is local-Cohen-Macaulay for each $\bar{\p} \in \mrm{Spec}(\mrm{H}^0(A))$. We write $\mrm{CM}(A)$ for the Cohen-Macaulay locus of $A$ (i.e. those $\bar{\p} \in \mrm{Spec}(\mrm{H}^0(A))$ for which $A_{\bar\p}$ is Cohen-Macaulay), and $\mrm{LCM}(A)$ for the local-Cohen-Macaulay locus of $A$.

	We warn the reader that unlike in the case of rings, in general local-Cohen-Macaulay does not imply Cohen-Macaulay due to localization causing a drop in amplitude, see~\cite[Example 8.3]{ShCM}. However, when the spectrum of $\mrm{H}^0(A)$ is irreducible and $\mrm{H}^0(A)$ is catenary, the two definitions agree by combining~\cite[Corollary 8.7]{ShCM} and~\cite[Theorem 2.5]{shaul2021sequence}.

\begin{lem}\label{LCM_flat_descent}
	Let $f\colon (A, \bar\m) \to (B, \bar\n)$ be a flat, local map of commutative noetherian local DG-rings, and $\bar{\p} \in \mrm{Spec}(\mrm{H}^0(B))$. Write $\bar\q := \mrm{H}^0(f)^{-1}(\bar\p)$. If $B_{\bar\p}$ is local-Cohen-Macaulay then $A_{\bar\q}$ is local-Cohen-Macaulay; in other words, if $\bar\p \in \mrm{LCM}(B)$ then $\bar\q \in \mrm{LCM}(A)$. Moreover, if $\mrm{H}^0(A)$ and $\mrm{H}^0(B)$ are both catenary and have irreducible spectra, then $\bar\p \in \mrm{CM}(B)$ implies $\bar\q \in \mrm{CM}(A)$. 

\end{lem}
\begin{proof}
	Note that by definition, such an $f$ must be faithfully flat.
	By the invariance of local cohomology under base change we have an isomorphism $\msf{R}\Gamma_{\bar\q}(A_{\bar\q}) \Lotimes_{A_{\bar\q}} B_{\bar\p} \cong \msf{R}\Gamma_{\bar\p}(B_{\bar\p})$ in $\msf{D}(B_{\bar\p})$. Therefore by \cref{prop:ffamp}, we have
	\[\amp(\msf{R}\Gamma_{\bar\q}(A_{\bar\q})) = \amp(\msf{R}\Gamma_{\bar\p}(B_{\bar\p})) = \amp(B_{\bar{\p}}) = \amp(A_{\bar\q})\] as required. 
\end{proof}

Building on recent results on the loci of Cohen-Macaulay DG-rings~\cite{ShLoci}, the following lemma provides a DG-version of~\cite[Proposition 5.2]{Bass}.
\begin{lem}\label{existence_CM_curves}
Let $A$ be a commutative noetherian DG-ring with bounded cohomology such that $1\leq d = \mrm{dim}(\mrm{H}^0(A)) \leq \infty$. Then for any natural number $1 \leq n \leq d$, there exists $\bar{\p} \in \mrm{Spec}(\mrm{H}^0(A))$ such that $\bar{\p} \in \mrm{LCM}(A) \cap \mrm{CM}(\mrm{H}^0(A))$, $\mrm{ht}(\bar{\p})=n-1$ and $\dim(\mrm{H}^0(A)/\bar{\p})\geq 1$.
\end{lem}
\begin{proof}
By localizing $A$ at a height $n$ prime of $\mrm{H}^0(A)$, we may assume $(A,\bar{\m})$ is local and $\dim(\mrm{H}^0(A))=n$. We write $\widehat{A}:= \msf{L}\Lambda(A,\bar{\m})$
for the derived $\bar{\m}$-completion of $A$. Note that ~\cite[Proposition 4.16]{Shaul2019} tells us that $\mrm{H}^0(\widehat{A})=\widehat{\mrm{H}^0(A)}$ and $\mrm{H}^{0}(A)\to \mrm{H}^{0}(\widehat{A})$ is the usual $\m$-adic completion map of rings. To obtain the prime $\bar{\p}$ as in the statement we use induction to construct a strict chain of length $n$  
\[\bar{\q}_0\subset \bar{\q}_1 \subset \dots \bar{\q}_i \subset \dots \subset \bar{\q}_n\] 
comprised of prime ideals in $\mrm{LCM}(\widehat{A})\cap \mrm{CM}(\mrm{H}^0(\widehat{A}))$ such that $\mrm{ht}(\bar{\q}_i\cap \mrm{H}^0(A))=i$ for all $i$. 

To this end, let us choose a strict chain of length $n$ consisting of prime ideals in $\mrm{H}^0(A)$, say
\[\bar{\mrm{P}}_0\subset \bar{\mrm{P}}_1\subset \dots\subset \bar{\mrm{P}}_n=\bar{\m}.\] 
Since $\widehat{\mrm{H}^0(A)}$ is faithfully flat over $\mrm{H}^0(A)$, going down holds for the extension, so there is a strict chain of length $n$ of prime ideals 
\[\bar{\mrm{Q}}_0\subset \dots \subset \bar{\mrm{Q}}_n=\widehat{\bar{\m}}\] 
in $\widehat{\mrm{H}^0(A)}$ which lies over the given chain in $\mrm{H}^0(A)$. Clearly, $\bar{\mrm{Q}}_0\in \mrm{LCM}(\widehat{A})\cap \mrm{CM}(\mrm{H}^0(\widehat{A}))$ and $\bar{\mrm{P}}_0=\bar{\mrm{Q}}_0\cap A$ has height zero. Set $\bar{\q}_0:=\bar{\mrm{Q}}_0$. By \cite[Proposition 7.21]{ShINJ} the DG-ring $\widehat{A}$ has a dualizing DG-module, so we may apply \cite[Theorem 12]{ShLoci} to choose an open neighbourhood $D(f)$ of $\bar{\q}_0$ in $\mrm{LCM}(\widehat{A}) \cap\mrm{CM}(\mrm{H}^0(\widehat{A}))$ which completes the base case.

For the inductive hypothesis, suppose that for some $1\leq w\leq n-1$, we have a strict chain of primes \[\bar{\q}_0\subset \dots \bar{\q}_{w-1}\subset  \bar{\mrm{Q}}_w\subset \dots \bar{\mrm{Q}}_n=\widehat{\bar{\m}}\] in $\widehat{\mrm{H}^0(A)}$ such that there is an open neighbourhood of $\bar{\q}_{w-1}$, say $D(f)$, contained in $\mrm{LCM}(\widehat{A})\cap\mrm{CM}(\mrm{H}^0(\widehat{A}))$ and $\mrm{ht}(\bar{\p}_i) = i$ where $\bar{\p}_i$ are defined by
\[
\bar{\p}_i :=
\begin{cases}
\bar{\q}_i\cap \mrm{H}^0(A) & \text{if $0 \leq i \leq w-1$,} \\
\bar{\mrm{Q}}_i \cap \mrm{H}^0(A) & \text{if $w \leq i \leq n$.}
\end{cases}
\]
Let us define \[\mathsf{X}_w = \{\bar\p \in \mrm{Spec}(\mrm{H}^0(A)) \mid \bar\q_{w-1} \subseteq \bar\p \subseteq \bar{\mrm{Q}}_{w+1}, \mrm{ht}(\bar{\p}) = w \text{ and } f \in \bar\p\}.\] By faithful flatness we have \[\mrm{ht}(\bar{\p}_{w+1}\widehat{\mrm{H}^0(A)})=\mrm{ht}(\bar{\p}_{w+1})=w+1,\] so $\bar{\p}_{w+1}$ is not contained in $\msf{X}_w$. Therefore we may choose an element $t\in \bar{\p}_{w+1}$ not contained in $\msf{X}_w$ and a height $w$ prime containing $t$, say $\bar{\q}_w$, which lies between $\bar{\q}_{w-1}$ and $\bar{\mrm{Q}}_{w+1}$. 

By choice, $\bar{\q}_w\in D(f)$ and therefore $\bar{\q}_w\in\mrm{LCM}(\widehat{A})\cap\mrm{CM}(\mrm{H}^0(\widehat{A}))$. Moreover, we have $\bar{\p}_{w-1}\subseteq \bar{\q}_w\cap \mrm{H}^0(A)\subseteq \bar{\p}_{w+1}$. But these inclusions are actually strict since going down ensures \[\mrm{ht}(\bar{\q}_w\cap \mrm{H}^0(A))\leq \mrm{ht}(\bar{\q}_w)=w,\] so $\bar{\q}_w\cap \mrm{H}^0(A)\neq \bar{\p}_{w+1}$. On the other hand, $t\in \bar{\q}_w\cap \mrm{H}^0(A)$, but $t\notin \bar{\p}_{w-1}$. Thus, $\mrm{ht}(\bar{\q}_w\cap \mrm{H}^0(A))=w$. Replacing $\bar{\mrm{Q}}_w$ with $\bar{\q}_w$ and $\bar{\p}_w$ with $\bar{\q}_w\cap \mrm{H}^0(A)$ concludes the induction.

As $A\rightarrow \widehat{A}$ is flat by~\cite[Corollary 4.6]{ShHomDim}, we can now use \cref{LCM_flat_descent} to conclude that $\bar{\q}_i \cap \mrm{H}^0(A)\in \mrm{LCM}(A)\cap \mrm{CM}(\mrm{H}^0(A))$, which is what we wanted to show.
\end{proof}

Before coming to the main result of this section, we interject with a brief introduction to Koszul complexes. Given an element $a \in \mrm{H}^0(A)$, we define \[K(A;a) = \mrm{cone}(A \xrightarrow{\cdot a} A)\] and given a sequence $\underline{a} = a_1,\ldots,a_n$ we write $K(A;\underline{a}) = K(A;a_1) \otimes_A \cdots \otimes_A K(A;a_n)$. We will write $\ell(\underline{a})$ for the length of the sequence $\underline{a}$. It follows easily from the definition that $K(A;\underline{a})$ is a compact object in $\msf{D}(A)$ and that moreover it is self dual up to a shift. 

If $(A,\bar{\m})$ is a noetherian local DG-ring with bounded cohomology, a sequence $\underline{a}=a_1,\dots,a_n \in \bar{\m}$ is $A$-regular if and only if
$\inf(K(A;\underline{a})) = \inf(A)$. The maximal length of $A$-regular sequence is called the sequential depth of $A$,
and is denoted by $\seqdepth_A(A)$. One can similarly define the notion of an $M$-regular sequence,
and the sequential depth of $M$ for any $M \in \cat{D}^{\mrm{b}}_{\mrm{f}}(A)$.
We refer the reader to \cite{Minamoto,ShCM,ShKoszul} for details.

The following lemma will be used in the proof of the subsequent theorem.

\begin{lem}\label{lem:infsupKoszul}
Let $A$ be a commutative noetherian DG-ring, $M \in \msf{D}^+(A)$ and $N \in \msf{D}^{-}(A)$. Then 
\[\inf(K(A;\underline{a}) \Lotimes_A M) \geq \inf(M) - \ell(\underline{a}) \quad \mrm{and} \quad \sup(\RHom_A(K(A;\underline{a}),N)) \leq \sup(N) + \ell(\underline{a}).\]
\end{lem}
\begin{proof}
Since the Koszul complex is compact and self-dual up to a shift, we have an isomorphism \[K(A;\underline{a}) \Lotimes_A M \cong \RHom_A(K(A;\underline{a}), M)[\ell(\underline{a})]\] in $\msf{D}(A)$. Therefore using \cref{lem:ExtLowerBound} we have \begin{align*}\inf(K(A;\underline{a}) \Lotimes_A M) &= \inf(\RHom_A(K(A;\underline{a}), M)) - \ell(\underline{a}) \\ &\geq \inf(M)-\sup(K(A;\underline{a})) -\ell(\underline{a})\\ &= \inf(M) - \ell(\underline{a})\end{align*}
proving the first inequality. A similar argument gives the second inequality.
\end{proof}

The following classical fact was proved in \cite[Proposition 5.4]{Bass}.
\begin{lem}\label{projdimlocalization}
	Let $A$ be a commutative noetherian ring of positive dimension. Let $P\subseteq A$  be a prime of positive codimension and $s\in A\setminus P$ such that $(P,s)\neq A$. If $a_1,\dots,a_n\in P$ are such that there images in $A_s$ form a regular sequence, then $\projdim_A(A_s/(a_1,\dots,a_n)A_s)=n+1$. 
\end{lem}

We are now in a position to state and prove the main result of this section, which yields DG-modules of a prescribed projective dimension.

\begin{thm}\label{thm:construction}
	Let $A$ be a commutative noetherian DG-ring with bounded cohomology such that $d = \dim(\mrm{H}^0(A)) \leq \infty$.
	Then for any natural number $0 \le n \le d$,
	there exists $M \in \cat{D}^{\mrm{b}}(A)$ with $\sup(M) = 0$,
	$\inf(M) \ge \inf(A)$ and $\projdim_A(M) = n$.
\end{thm}
\begin{proof}
	If $n=0$, $M=A$ will do. So assume $n\geq 1$.
	Choose a prime $\bar{\p}\subseteq \mrm{H}^0(A)$ as in \cref{existence_CM_curves}. By \cite[Corollary 5.21]{ShCM}, there exists a sequence $\underline{a}=a_1,\dots, a_{n-1}\in \mrm{H}^0(A)$ whose image in $\mrm{H}^0(A)_{\bar{\p}}$ forms an $A_{\bar{\p}}$ regular sequence while simultaneously being a system of parameters for $\mrm{H}^0(A)_{\bar{\p}}$. Since $\bar{\p}\in \mrm{LCM}(\mrm{H}^0(A))$, the image of this sequence in $\mrm{H}^0(A)_{\bar{\p}}$ is also $\mrm{H}^0(A)_{\bar{\p}}$-regular.
	\par 
	Since taking Koszul complexes commutes with localization, it follows that the points in $\bar{\q} \in \opn{Spec}(\mrm{H}^0(A))$ at which the images of the $a_i$ form a regular sequence on $A_{\bar{\q}}$ and $H^0(A)_{\bar{\q}}$ is an open set containing $\bar{\p}$. Take an open neighbourhood $D(f)$ of $\bar{\p}$ contained within this open set. By hypothesis, there exists $x\in \mrm{H}^0(A)\setminus \bar{\p}$ such that $(\bar{\p},x)\neq \mrm{H}^0(A)$. Let $s$ denote any lift to $A^0$ of $xf\in \mrm{H}^0(A)$. By construction, the images of the $a_i$ in $\mrm{H}^0(A)_{xf}$ form a regular sequence on $A_s$ and $\mrm{H}^0(A)_{xf}$.
	\par Consider the Koszul DG-ring $K(A_s;a_1,\dots,a_{n-1})$ of $A_s$ 
with respect to $a_1,\dots,a_{n-1}$. As $a_1,\dots,a_{n-1}$ is $A_s$-regular, we have that $\amp(K(A_s;a_1,\dots,a_{n-1})) = \amp(A_s) \leq\amp(A)$. From the definition one sees that $\mrm{H}^0(A) \Lotimes_A K(A_s;a_1,\dots,a_{n-1}) \simeq K(\mrm{H}^0(A)_{xf};a_1,\dots,a_{n-1})$ and therefore
\[\projdim_A(K(A_s;a_1,\dots,a_{n-1})) = \projdim_{\mrm{H}^0(A)}(K(\mrm{H}^0(A)_{xf};a_1,\dots,a_{n-1})) = n\]
by \cref{cor:injdimcoreduction} and \cref{projdimlocalization}. Consequently, taking $M := K(A_s;a_1,\dots,a_{n-1})$ provides the DG-module required for the statement.
\end{proof}

The final goal of this section is calculating the flat dimension of the module $M$ constructed in the previous theorem. In order to do this, we require the following.
\begin{lem}\label{prop:fg_flat_proj}
Let $A$ be a commutative noetherian DG-ring. For any $M \in \msf{D}^-_\mrm{f}(A)$ we have $\projdim_A(M) = \flatdim_A(M)$.
\end{lem}
\begin{proof}
Since $\mrm{H}^{0}(A)\Lotimes_{A}M\in\D_\mrm{f}^\mrm{-}(\mrm{H}^{0}(A))$ by \cref{lem:infcoreduction}, by applying~\cite[Corollary 2.10.F]{AF} we have \[\flatdim_{\mrm{H}^{0}(A)}(\mrm{H}^0(A) \Lotimes_A M)=\projdim_{\mrm{H}^{0}(A)}(\mrm{H}^0(A) \Lotimes_A M).\] In particular, we see that $\projdim_{A}(M)=\flatdim_{A}(M)$ by \cref{cor:injdimcoreduction}.
\end{proof}

The following result shows that the projective dimension of Koszul complexes is as expected. 

\begin{lem}\label{prop:projdimKoszul}
Let $A$ be a commutative DG-ring, and let $\underline{a}$ be a finite sequence of elements in $\mrm{H}^0(A)$ which generates a proper ideal. 
Then 
$\projdim_A(K(A;\underline{a})) = \ell(\underline{a}).$
\end{lem}
\begin{proof}
For any $N \in \msf{D}^{\mrm{b}}(A)$ we have 
$\sup(\RHom_A(K(A;\underline{a}),N)) \leq \sup(N) + \ell(\underline{a})$ by~\cref{lem:infsupKoszul}.
This shows that $\projdim_A(K(A;\underline{a})) \leq \ell(\underline{a})$. 
On the other hand, note that since
\[
\RHom_A(K(A;\underline{a}),A) = K(A;\underline{a})[-\ell(\underline{a})]
\]
it follows that $\opn{Ext}^{\ell(\underline{a})}_A(K(A;\underline{a}),A) = \mrm{H}^0(A)/\underline{a} \ne 0$,
and this shows that $\projdim_A(K(A;\underline{a})) \ge \ell(\underline{a})$ as $\sup(A) = 0$.
\end{proof}

Using the above results, we can calculate the flat dimension of the module constructed in \cref{thm:construction}.
\begin{prop}\label{cor:FFD_lowerBD}
Let $A$ be a commutative noetherian DG-ring with bounded cohomology such that $1\leq d=\dim(\mrm{H}^0(A))\leq \infty$. Then, for any natural number $1\leq n\leq d$, there exists $M \in \cat{D}^{\mrm{b}}(A)$ with $\flatdim_A(M)=n-1$, $\sup(M)=0$ and $\inf(M)\ge\inf(A)$.
\end{prop}
\begin{proof}
If $n=1$, take $M=A$. So we fix $n\geq 2$ and consider the module $M \in \cat{D}^{\mrm{b}}(A)$ constructed in \cref{thm:construction} for this choice of $n$. With notation as in the proof of \cref{thm:construction}, since $M\in \msf{D}^\mrm{b}_\mrm{f}(A_s)$, we have $\flatdim_{A_s}(M)=\projdim_{A_s}(M)=n-1$ by applying \cref{prop:fg_flat_proj} and \cref{prop:projdimKoszul} in turn. Since $A_s$ is $A$-flat, it follows that $\flatdim_A(M)=n-1$.
\end{proof}

\section{Finitistic dimensions and Bass's questions}\label{sec:fpd}
In this section we collate the bounds obtained so far to approach Bass's two questions for the finitistic projective dimensions over commutative noetherian DG-rings with bounded cohomology. Afterwards, we also consider the finitistic dimensions with respect to flat and injective dimension.

\subsection{Finitistic projective dimensions} 
Firstly we recall the dimensions of the small and the big finitistic projective dimensions.
\begin{dfn}\label{deffpd}
Let $A$ be a commutative DG-ring.
\begin{itemize}
\item[(i)] The \emph{small finitistic projective dimension} of $A$ is defined by \[\msf{fpd}(A) = \sup\{\projdim_A(M) + \inf(M) \mid M \in \msf{D}^\mrm{b}_\mrm{f}(A),~\projdim_A(M) < \infty\}.\]
\item[(ii)] The \emph{big finitistic projective dimension} of $A$ is defined by \[\msf{FPD}(A) = \sup\{\projdim_A(M) + \inf(M) \mid M \in \msf{D}^\mrm{b}(A),~\projdim_A(M) < \infty\}.\]
\end{itemize}
\end{dfn}

We now state our bounds on the big finitistic projective dimension, which is the cumulation of our efforts so far.

\begin{thm}\label{cor:FPDbounds}
Let $A$ be a commutative noetherian DG-ring with bounded cohomology, and suppose that $\mrm{dim}(\mrm{H}^0(A)) < \infty$. Then \[\mrm{dim}(\mrm{H}^0(A)) - \amp(A) \leq \msf{FPD}(A) \leq \mrm{dim}(\mrm{H}^0(A)).\]
\end{thm}
\begin{proof}
The upper bound for $\msf{FPD}(A)$ follows from \cref{thm:globalfpdbound}. By \cref{thm:construction}, there exists a DG-module $M \in \msf{D}^\mrm{b}(A)$ with $\projdim_A(M) = \mrm{dim}(\mrm{H}^0(A))$ and $\inf(M) \geq -\amp(A)$ which proves the lower bound.
\end{proof}

\begin{rem}
Let us contrast the above bounds to those in the case of modules, or even complexes, over commutative noetherian rings. We highlight that the bounds in the above theorem cannot be improved, as is shown in \cref{sec:examples}. In particular, we may find commutative noetherian DG-rings with bounded cohomology whose big finitistic projective dimension is strictly less than $\mrm{dim}(\mrm{H}^{0}(A))$. This is a stark difference to the case of commutative noetherian rings. If $A$ is an ordinary commutative noetherian ring, then it is known (see, for example, \cite[E8.5.4]{dcmca}) that the finitistic projective dimension of $A$ (defined as in~\cref{deffpd}) is equal to $\dim(A)$. Note that this is incorporated into the preceding theorem, simply by considering the case when the DG-ring has amplitude zero.
\end{rem}

We now turn to identifying the small finitistic projective dimension. To begin with, we treat the local case.  
The following is an easy consequence of the DG-Auslander-Buchsbaum formula~\cite[Theorem 2.16]{Minamoto}.
\begin{lem}\label{ABformula}
Let $(A,\bar{\m})$ be a commutative noetherian local DG-ring with bounded cohomology, and $M \in \msf{D}^\mrm{b}_\mrm{f}(A)$. 
If $\projdim_A(M) < \infty$,
then $\projdim_A(M) \leq \seqdepth_A(A) - \inf(M) - \amp(A)$. 
\end{lem}

Using \cref{ABformula} we may now determine the small finitistic projective dimension for local DG-rings.

\begin{prop}\label{smallfindimdepth}
Let $(A,\bar{\m})$ be a commutative noetherian local DG-ring with bounded cohomology. Then $\msf{fpd}(A) = \seqdepth_A(A) - \amp(A)$.
\end{prop}

\begin{proof}
We have that $\msf{fpd}(A) \leq \seqdepth_A(A) - \amp(A)$ by \cref{ABformula}. 
The converse inequality holds by taking an $A$-regular sequence of length $\seqdepth_A(A)$ and applying \cref{prop:projdimKoszul},
since if $\underline{a}$ is $A$-regular then $\inf(K(A;\underline{a})) = \inf(A) = -\amp(A)$.
\end{proof}

By combining the previous result with~\cite[Corollary 5.5]{ShCM} we obtain the following characterization of local-Cohen-Macaulay DG-rings.
\begin{cor}\label{cmFPD}
Let $(A,\bar{\m})$ be a commutative noetherian local DG-ring with bounded cohomology.
Then $A$ is local-Cohen-Macaulay if and only if $\msf{fpd}(A) = \dim(\mrm{H}^0(A)) - \amp(A)$.
\end{cor}

\begin{rem}
As we stated in the introduction, a commutative noetherian local ring $A$ is Cohen-Macaulay if and only if $\msf{fpd}(A)=\msf{FPD}(A)$. We are now in a position to consider whether this characterisation also holds in the world of DG-rings. Unfortunately, the question remains open: if $(A,\bar{\m})$ is a local-Cohen-Macaulay DG-ring, then \cref{cmFPD} tells us that $\msf{fpd}(A)=\dim(\mrm{H}^0(A)) - \amp(A)$. However, \cref{cor:FPDbounds} tells us that we may only deduce that $\msf{fpd}(A)\leq \msf{FPD}(A)\leq \mrm{dim}(\mrm{H}^{0}(A))$. In \cref{cor:Gorenstein} we consider a special case regarding certain Gorenstein DG-rings, which provides some but not conclusive evidence for the analogous statement in the DG-world.
\end{rem}

We now turn to the global case, which is our final result on finitistic projective dimension.
\begin{prop}\label{lem:projdimlocal}
Let $A$ be a commutative noetherian DG-ring. For any $M \in \msf{D}^{-}_\mrm{f}(A)$ we have \[\projdim_A(M) = \sup\{\projdim_{A_{\bar\p}}(M_{\bar\p}) \mid \bar\p \in \mrm{Spec}(\mrm{H}^0(A))\}.\]
Consequently $\msf{fpd}(A) \leq \sup\{\msf{fpd}(A_{\bar\p}) \mid \bar\p \in \mrm{Spec}(\mrm{H}^0(A))\}$.
\end{prop}
\begin{proof}
By applying \cref{cor:injdimcoreduction} and~\cite[Proposition 5.3.P]{AF}, we have 
\begin{align*}
\projdim_A(M) = \sup\{\projdim_{\mrm{H}^0(A)}((\mrm{H}^0(A) \Lotimes_A M)_{\bar\p}) \mid \bar\p \in \mrm{Spec}(\mrm{H}^0(A)\}. \end{align*}
Since $(\mrm{H}^0(A) \Lotimes_A M)_{\bar\p} \cong \mrm{H}^0(A) \Lotimes_A M_{\bar\p}$, another application of \cref{cor:injdimcoreduction} completes the proof of the first part of the statement. The second part of the statement follows, since $\inf(M_{\bar\p}) \geq \inf(M)$.
\end{proof}

\subsection{Finitistic flat and injective dimensions} 
We now investigate the finitistic flat and injective dimensions, which we will see are very closely related. Both are defined similarly to the projective case. We first consider the big finitistic dimensions.
\begin{dfn}
Let $A$ be a commutative DG-ring.
\begin{itemize}
\item[(i)] The \emph{big finitistic flat dimension} of $A$ is defined by \[\msf{FFD}(A) = \sup\{\flatdim_A(M) + \inf(M) \mid M \in \msf{D}^\mrm{b}(A),~\flatdim_A(M) < \infty\}.\]
\item[(ii)] The \emph{big finitistic injective dimension} of $A$ is defined by \[\msf{FID}(A) = \sup\{\injdim_A(M) - \sup(M) \mid M \in \msf{D}^\mrm{b}(A),~\injdim_A(M) < \infty\}.\]
\end{itemize}
\end{dfn}

Firstly, we give bounds for $\msf{FFD}(A)$ using the results of the previous sections.
\begin{lem}\label{FFDbounds}
Let $A$ be a commutative noetherian DG-ring with bounded cohomology and $\mrm{dim}(\mrm{H}^0(A)) < \infty$. Then
\[\mrm{dim}(\mrm{H}^0(A)) - \amp(A) - 1 \leq \msf{FFD}(A) \leq \mrm{dim}(\mrm{H}^0(A)).\]
\end{lem}
\begin{proof}
It is clear that $\msf{FFD}(A) \leq \msf{FPD}(A)$ and so we obtain the upper bound from \cref{cor:FPDbounds}. The lower bound follows from \cref{cor:FFD_lowerBD}.
\end{proof}

We now show that the big finitistic injective dimension and the big finitistic flat dimension are equal.
Fix a faithfully injective $\mrm{H}^0(A)$-module $\bar{E}$. By Brown representability, there exists an object $E\in\D(A)$ such that for all $M \in \msf{D}(A)$ we have
\begin{itemize}
\item[(i)] $\RHom_{A}(M,E)\cong 0$ if and only if $M \cong 0$; and,
\item[(ii)] $\mrm{H}^{i}(\RHom_{A}(M,E))\cong \Hom_{\mrm{H}^{0}(A)}(\mrm{H}^{-i}(M),\bar{E})$ for all $i\in\mbb{Z}$. 
\end{itemize} 
For brevity we write $\mbb{I}M = \RHom_A(M,E)$. We refer the reader to~\cite[Section 3.1]{dualitypairs} and \cite{ShINJ} for more details on this construction. It follows from (ii) that we have the equalities
\begin{equation}\label{eqn:BCsup}
\sup(\mbb{I}M) = -\inf(M) \quad \text{and} \quad \inf(\mbb{I}M) = -\sup(M)
\end{equation}
for all $M \in \msf{D}(A)$. 

\begin{lem}\label{BCflatinj}
Let $A$ be a commutative noetherian DG-ring.
\begin{itemize}
	\item[(i)] For any $M \in \msf{D}(A)$, $\flatdim_{A}(M)=\injdim_{A}(\RHom_{A}(M,E))$.
	\item[(ii)] For any $M \in \msf{D}^+(A)$, $\injdim_{A}(M)=\flatdim_{A}(\RHom_{A}(M,E))$.
\end{itemize}
\end{lem}
\begin{proof}
For any ideal $\bar{I} \subseteq \mrm{H}^0(A)$,
by adjunction there is an isomorphism
\[
\RHom_{A}(\mrm{H}^{0}(A)/\bar{I},\RHom_{A}(M,E))\cong \RHom_{A}(\mrm{H}^{0}(A)/\bar{I}\Lotimes_{A}M,E)
\]
for any $M\in\D(A)$. In particular, this isomorphism and the fact that $\mrm{H}^{i}(\RHom_{A}(M,E))\cong \Hom_{\mrm{H}^{0}(A)}(\mrm{H}^{-i}(M),\bar{E})$ proves (i). For part (ii), we have
\begin{align*}
\text{Tor}_{i}^{A}(\mrm{H}^{0}(A)/\bar{I},\RHom_{A}(M,E)) &= \mrm{H}^{-i}(\mrm{H}^{0}(A)/\bar{I}\Lotimes_{A}\RHom_{A}(M,E))\\
&\cong \mrm{H}^{-i}(\RHom_{A}(\RHom_{A}(\mrm{H}^{0}(A)/\bar{I},M),E)) \mbox{ by \cref{tensorhomeval}}\\
&\cong \Hom_{\mrm{H}^{0}A}(\text{Ext}_{A}^{i}(\mrm{H}^{0}(A)/\bar{I},M),\bar{E}).
\end{align*}
In particular, using \cref{thm:injForm} and its analogue in the flat case, we see that $\flatdim_A(\RHom_{A}(M,E)) = \injdim_A(M)$ as required.
\end{proof}

\begin{rem}
In other words, the previous result shows that $(\msf{F}_n, \msf{I}_n)$ is symmetric duality pair in the sense of~\cite{dualitypairs}, where $\msf{F}_n$ (resp., $\msf{I}_n$) consists of those $M \in \msf{D}^\mrm{b}(A)$ for which $\flatdim_A(M) \leq n$ (resp., $\injdim_A(M) \leq n$).
\end{rem}

Using the previous result we obtain the following.
\begin{prop}\label{cor:FIDequalFFD}
Let $A$ be a commutative noetherian DG-ring. Then $\msf{FFD}(A)=\msf{FID}(A)$.
\end{prop}
\begin{proof}
We may assume that $\msf{FFD}(A)$ is finite, otherwise, it follows that $\msf{FID}(A)$ is also infinite from \cref{BCflatinj}. Suppose $M \in \msf{D}^\mrm{b}(A)$ has $\flatdim_A(M) + \inf(M) = \msf{FFD}(A).$ Then \[\injdim_{A}(\mbb{I}M) - \sup(\mbb{I}M) =\msf{FFD}(A)\] by  \cref{eqn:BCsup} and \cref{BCflatinj}, and hence $\msf{FID}(A) \geq \msf{FFD}(A)$. We now suppose that this inequality is strict and derive a contradiction. If there exists a $N \in \msf{D}^\mrm{b}(A)$ with finite injective dimension and $\injdim_A(N) - \sup(N) > \msf{FFD}(A)$, then \[\flatdim_A(\mbb{I}N) + \inf(\mbb{I}N) = \injdim_A(N) > \msf{FFD}(A)\] by \cref{eqn:BCsup} and \cref{BCflatinj}. This contradicts the definition of $\msf{FFD}(A)$ so $\msf{FFD}(A) = \msf{FID}(A)$ as required.
\end{proof}

As such, we obtain bounds on $\msf{FID}(A)$ from \cref{FFDbounds}; namely, we have
\[\mrm{dim}(\mrm{H}^0(A)) - \amp(A) - 1 \leq \msf{FID}(A) \leq \mrm{dim}(\mrm{H}^0(A))\]
for any commutative noetherian DG-ring $A$ with bounded cohomology and $\mrm{dim}(\mrm{H}^0(A)) < \infty$.

We now turn to understanding the small finitistic flat and injective dimensions. 
\begin{dfn}
Let $A$ be a commutative DG-ring.
\begin{itemize}
\item[(i)] The \emph{small finitistic flat dimension} of $A$ is defined by \[\msf{ffd}(A) = \sup\{\flatdim_A(M) + \inf(M) \mid M \in \msf{D}^\mrm{b}_\mrm{f}(A),~\flatdim_A(M) < \infty\}.\]
\item[(i)] The \emph{small finitistic injective dimension} of $A$ is defined by \[\msf{fid}(A) = \sup\{\injdim_A(M) -\sup(M) \mid M \in \msf{D}^\mrm{b}_\mrm{f}(A),~\injdim_A(M) < \infty\}.\]
\end{itemize}
\end{dfn}

The next lemma is, essentially, the be-all and end-all for the small finitistic flat and injective dimensions in the local case.

\begin{prop}\label{fpdequalffd}
Let $(A, \bar{\m})$ be a commutative noetherian local DG-ring with bounded cohomology. Then all three of $\msf{fpd}(A)$, $\msf{ffd}(A)$ and $\msf{fid}(A)$ agree, and are equal to $\seqdepth_A(A)-\mrm{amp}(A)$.
\end{prop}

\begin{proof}
The equality $\msf{fpd}(A)=\msf{ffd}(A)$ is \cref{prop:fg_flat_proj}, which are equal to $\seqdepth_A(A)-\mrm{amp}(A)$ by \cref{smallfindimdepth}. For the injective case, let $M\in\D_\mrm{f}^\mrm{b}(A)$ have finite injective dimension. By the Bass formula given in~\cite[Theorem 3.33]{Minamoto}, we have that $\injdim_{A}(M)-\sup(M)=\seqdepth_A(A)-\amp(A)$. Therefore by the definition of $\msf{fid}(A)$, it is clear that $\msf{fid}(A)=\seqdepth_A(A)-\text{amp}(A)$.
\end{proof}

We note that the proof shows that any $M\in\D_\mrm{f}^\mrm{b}(A)$ with finite injective dimension achieves the bound. 

We also have the following results for non-local DG-rings, whose proofs are analogous to that of \cref{lem:projdimlocal}, using~\cite[Propositions 5.3.F and 5.3.I]{AF}.

\begin{lem}\label{lem:flatdimlocal}
Let $A$ be a commutative noetherian DG-ring. 
\begin{enumerate}
\item For any $M \in \msf{D}(A)$ we have \[\flatdim_A(M) = \sup\{\flatdim_{A_{\bar\p}}(M_{\bar\p}) \mid \bar\p \in \mrm{Spec}(\mrm{H}^0(A))\}.\] Consequently $\msf{ffd}(A) \leq \sup\{\msf{ffd}(A_{\bar\p}) \mid \bar\p \in \mrm{Spec}(\mrm{H}^0(A))\}$ and similarly for $\msf{FFD}(A)$. 

\item For any $M \in \msf{D}^+(A)$ we have \[\injdim_A(M) = \sup\{\injdim_{A_{\bar\p}}(M_{\bar\p}) \mid \bar\p \in \mrm{Spec}(\mrm{H}^0(A))\}.\] Consequently $\msf{fid}(A) \leq \sup\{\msf{fid}(A_{\bar\p}) \mid \bar\p \in \mrm{Spec}(\mrm{H}^0(A))\}$ and similarly for $\msf{FID}(A)$. 
\end{enumerate}
\end{lem}

\section{Proving optimality of the bounds}\label{sec:examples}
In \cref{cor:FPDbounds} we showed that if $A$ is a commutative noetherian DG-ring with bounded cohomology such that $\dim(\mrm{H}^0(A)) < \infty$ then the big finitistic projective dimension satisfies
\[
\dim(\mrm{H}^0(A)) - \amp(A) \le \msf{FPD}(A) \le \dim(\mrm{H}^0(A)).
\]
The aim of this section is to construct examples
showing that this result cannot be improved.
In fact, we show there are examples of DG-rings $A$ for which $\msf{FPD}(A) = \dim(\mrm{H}^0(A)) - \amp(A)$,
whilst there are other examples of $A$ for which $\msf{FPD}(A) = \dim(\mrm{H}^0(A))$. 
Note that for any commutative noetherian ring $A$ of finite Krull dimension,
it holds that $\msf{FPD}(A) = \dim(A)$, but
the examples we construct involve DG-rings of positive amplitude. In fact, we construct whole families of examples where the dimension of $\mrm{H}^0(A)$ and the amplitude of $A$ range independently across the whole non-negative integers. 

For all integers $d\geq 0$ and $n>0$, we first construct examples of commutative noetherian DG-rings $A$ of amplitude $n$ such that $\msf{FPD}(A) = \dim(\mrm{H}^0(A))=d$.

\begin{exa}\label{ex:upper_bound}
Let $R:=B\times C$ be a product of noetherian rings and suppose without loss of generality that $\mrm{dim}(C) \leq \mrm{dim}(B) = d$. Then \[\dim(R)=\text{max}\{\dim(B),\dim(C)\}=d.\] Take $A$ to be the trivial extension DG-ring $A = R\skewtimes C[n]$; see \cite[Section 1]{JorDual} for details. Note that $A$ is a commutative noetherian DG-ring with $\mrm{H}^0(A) = R$,
$\mrm{H}^{-n}(A) = C$, and $\mrm{H}^i(A) = 0$ for all $i\notin \{-n,0\}$. In particular, $\amp(A) = n$ and $\dim(\mrm{H}^0(A)) = d$. 

Let $M$ be a $B$-module such that $\projdim_B(M)=d$; such a module exists by a result of Bass contained in \cref{thm:construction}. Choose a $B$-module $N$ such that $\opn{Ext}^d_B(M,N)\neq 0$. Since $B$ is a localization of $A$ and $M$ is a $B$-module, we have an isomorphism
\[
\RHom_{A}(M,N) \cong \RHom_{B}(M,N)
\]
in $\msf{D}(A)$ by adjunction.
Therefore $\projdim_{A}(M)\geq d$. Since $M$ is clearly of finite flat (and projective) dimension over $A$, we have $d\geq \msf{FPD}(A)\geq \projdim_{A}(M)+\inf(M)\geq d $ by~\cref{cor:FPDbounds}, and hence $\msf{FPD}(A)=d$.
\end{exa}

Next, for all integers $d\geq 0$ and $n>0$, we will construct examples of commutative noetherian DG-rings $A$ of amplitude $n$ such that $\dim(\mrm{H}^0(A))=d$ and $\msf{FPD}(A)=\dim(\mrm{H}^0(A))-\amp(A)$.
To do this, we first prove a general result about certain Gorenstein DG-rings.
Recall, following \cite{FJ},
that a commutative noetherian local DG-ring $(A,\bar{\m})$ is called Gorenstein if $\injdim_A(A) < \infty$.
This implies that $\amp(A) < \infty$.

\begin{thm}\label{thm:GorensteinFPD}
Let $(A,\bar{\m})$ be a Gorenstein local DG-ring.
Suppose that the noetherian local ring $\mrm{H}^0(A)$ is also Gorenstein.
Then for any $M \in \cat{D}^{\mrm{b}}(A)$ with $\flatdim_A(M) < \infty$
it holds that $\projdim_A(M) \le \dim(\mrm{H}^0(A)) -\amp(A) -\inf(M)$.
\end{thm}
\begin{proof}
Since $A$ and $\mrm{H}^0(A)$ are Gorenstein, 
it follows that $A$ is a dualizing DG-module over itself and that $\mrm{H}^0(A)$ is dualizing complex over itself.
By \cite[Theorem I]{FIJ} or \cite[Proposition 7.5]{YeDual} and the fact that over a local ring a dualizing complex is unique up to a shift,
it follows that $\RHom_A(\mrm{H}^0(A),A) \cong \mrm{H}^0(A)[k]$ for some $k \in \mathbb{Z}$.
By \cref{lem:infcoreduction} we know that 
\[
\inf(\RHom_A(\mrm{H}^0(A),A)) = \inf(A),
\]
so we deduce that
$\RHom_A(\mrm{H}^0(A),A) \cong \mrm{H}^0(A)[-\inf(A)]$.
Using this fact and \cref{tensorhomeval},
there are isomorphisms
\[
\RHom_A(\mrm{H}^0(A),M) \cong \RHom_A(\mrm{H}^0(A),A) \Lotimes_A M \cong \mrm{H}^0(A)[-\inf(A)] \Lotimes_A M.
\] 
These isomorphisms and another application of \cref{lem:infcoreduction} show that
\[
\inf(M) = \inf(\RHom_A(\mrm{H}^0(A),M)) = \inf(\mrm{H}^0(A)[-\inf(A)] \Lotimes_A M) = \inf(\mrm{H}^0(A) \Lotimes_A M) + \inf(A).
\]
Therefore we have
\begin{align*}
\projdim_A(M) &= \projdim_{\mrm{H}^0(A)}(\mrm{H}^0(A) \Lotimes_A M) & \mbox{ by \cref{cor:injdimcoreduction}}, \\
&\leq \dim(\mrm{H}^0(A))-\inf(\mrm{H}^0(A) \Lotimes_A M) & \mbox{ by \cref{thm:globalfpdbound}}, \\
&=\dim(\mrm{H}^0(A)) + \inf(A) - \inf(M) & \\ 
&= \dim(\mrm{H}^0(A)) - \amp(A) -\inf(M) &
\end{align*}
as required.
\end{proof}

By combining \cref{cor:FPDbounds} and \cref{thm:GorensteinFPD} we obtain the following corollary.
\begin{cor}\label{cor:Gorenstein}
Let $(A,\bar{\m})$ be a Gorenstein local DG-ring such that $\mrm{H}^0(A)$ is Gorenstein.
Then $\msf{FPD}(A)=\dim(\mrm{H}^0(A))-\amp(A)$.
\end{cor}

We now give a concrete example of a DG-ring to which the previous corollary applies, thus showing that for all integers $d\geq 0$ and $n>0$, there are commutative noetherian DG-rings $A$ of amplitude $n$ such that $\dim(\mrm{H}^0(A))=d$ and $\msf{FPD}(A)=\dim(\mrm{H}^0(A))-\amp(A)$.

\begin{exa}\label{ex:lower_bound}
To begin with, consider a Gorenstein local ring $B$ of dimension $d+k$ for some $k>0$ and let $b_1,\dots,b_k$ be a regular sequence in $B$ generating an ideal $I$. The quotient ring $B/I$ is also Gorenstein, and has dimension $d$. If we enlarge the generating set $b_1,\dots,b_k$ to a generating set $b_1,\dots,b_{n+k}$ for $I$ and set $A := K(B;b_1,\dots,b_{n+k})$, then $A$ is a commutative noetherian local DG-ring with bounded cohomology satisfying $\dim(\mrm{H}^0(A)) = d$ and $\amp(A)=n$.
By~\cite[Theorem 4.9]{FJ} or \cite{AvGol}, 
the fact that $B$ is a Gorenstein ring implies that $A$ is a Gorenstein DG-ring.
Moreover, $\mrm{H}^0(A) = B/I$ is also Gorenstein, and hence $\msf{FPD}(A)=\dim(\mrm{H}^0(A))-\amp(A)$ by \cref{cor:Gorenstein}.
\end{exa}

By combining the examples constructed in this section, we obtain the following theorem, which shows that for any dimension and any amplitude, we may obtain either of the bounds on $\msf{FPD}(A)$ given in \cref{cor:FPDbounds}. In particular, we see that in general neither amplitude, nor dimension, are sufficient to determine big finitistic projective dimension.
\begin{thm}\label{thm:examples}
Given integers $d\geq 0$ and $n>0$, there exist commutative noetherian DG-rings $A$ and $A'$ with $\amp(A) = \amp(A') = n$, such that $\dim(\mrm{H}^0(A))=\dim(\mrm{H}^0(A'))=d$, $\msf{FPD}(A)=\dim(\mrm{H}^0(A))-\amp(A)$ and $\msf{FPD}(A')=\dim(\mrm{H}^0(A'))$.
\end{thm}

\section{An application to homologically smooth maps}\label{sec:smooth}
In this section, we use the results of this paper to obtain vanishing results for derived Hochschild (co)homology of homologically smooth algebras. 
Recall that a homomorphism of commutative rings $A \to B$ is called \emph{homologically smooth}
if the ring $B$, considered as an object of $\cat{D}(B\Lotimes_A B)$ is compact.
This definition is due to Kontsevich, see \cite{KonSo} for details.

Recall that the derived Hochschild homology modules of a map of commutative rings $A \to B$ 
are defined to be the $B$-modules $\mrm{HH}_i(B/A) := \opn{Tor}_{i}^{B\Lotimes_A B}(B,B)$.
Similarly, the derived Hochschild cohomology modules of $B$ over $A$ are given by 
$\mrm{HH}^i(B/A) := \opn{Ext}^i_{B\Lotimes_A B}(B,B)$.
\begin{cor}\label{cor:smooth-bound}
Let $A$ and $B$ be commutative noetherian rings,
and let $\varphi\colon A \to B$ be a ring homomorphism.
Assume that $\varphi$ satisfies:
\begin{enumerate}
\item $\varphi$ is essentially of finite type, i.e., it makes $B$ into a localization of a finitely generated $A$-algebra;
\item the map $\varphi$ has finite flat dimension, i.e., $\flatdim_A(B) < \infty$.
\item the map $\varphi$ is homologically smooth.
\end{enumerate}
Then 
\[
\projdim_{B\Lotimes_A B}(B) \le \dim(B\otimes_A B).
\]
\end{cor}
\begin{proof}
Since $\varphi$ is homologically smooth,
it follows that $\projdim_{B\Lotimes_A B}(B) < \infty$.
The fact that $\varphi$ has finite flat dimension and is essentially of finite type implies that
$B\Lotimes_A B$ is a commutative noetherian DG-ring with bounded cohomology.
By \cref{thm:globalfpdbound}, 
we deduce that
\[
\projdim_{B\Lotimes_A B}(B) \le \dim(\mrm{H}^0(B\Lotimes_A B)) - \inf(B) = \dim(B\otimes_A B),
\]
where the latter follows from the fact that by definition
$\mrm{H}^0(B\Lotimes_A B) = B\otimes_A B.$
\end{proof}

\begin{cor}
In the situation of \cref{cor:smooth-bound},
the derived Hochschild homology and derived Hochschild cohomology modules satisfy
\[
\mrm{HH}_i(B/A) = \opn{Tor}_{i}^{B\Lotimes_A B}(B,B) = 0 \quad \text{and} \quad
\mrm{HH}^i(B/A) = \opn{Ext}^i_{B\Lotimes_A B}(B,B) = 0
\]
for all $i>\dim(B\otimes_A B)$.
\end{cor}

\bibliographystyle{abbrv}
\bibliography{references}

\begin{thebibliography}{10}

\bibitem{APThesis}
D.~Apassov.
\newblock Complexes and differential graded modules.
\newblock {\em Ph.D. thesis, Lund University}, 1999.

\bibitem{AF}
L.~L. Avramov and H.-B. Foxby.
\newblock Homological dimensions of unbounded complexes.
\newblock {\em J. Pure Appl. Algebra}, 71(2-3):129--155, 1991.

\bibitem{AvGol}
L.~L. Avramov and E.~S. Golod.
\newblock The homology of algebra of the {K}oszul complex of a local
  {G}orenstein ring.
\newblock {\em Mat. Zametki}, 9:53--58, 1971.

\bibitem{Bass}
H.~Bass.
\newblock Injective dimension in {N}oetherian rings.
\newblock {\em Trans. Amer. Math. Soc.}, 102:18--29, 1962.

\bibitem{dualitypairs}
I.~Bird and J.~Williamson.
\newblock Duality pairs, phantom maps, and definability in triangulated
  categories.
\newblock {\em Proc. Roy. Soc. Edinburgh Sect. A}.
\newblock Published online 2024:1-46.
  \url{https://doi.org/10.1017/prm.2024.73}.

\bibitem{dcmca}
L.~W. Christensen, H.-B. Foxby, and H.~Holm.
\newblock Derived category methods in commutative algebra.
\newblock Draft version of 11/21. Obtained at
  \url{https://www.math.ttu.edu/~lchriste/download/dcmca.pdf}.

\bibitem{Eisenbud}
D.~Eisenbud.
\newblock {\em Commutative algebra}, volume 150 of {\em Graduate Texts in
  Mathematics}.
\newblock Springer-Verlag, New York, 1995.
\newblock With a view toward algebraic geometry.

\bibitem{Enochs}
E.~Enochs.
\newblock Flat covers and flat cotorsion modules.
\newblock {\em Proc. Amer. Math. Soc.}, 92(2):179--184, 1984.

\bibitem{Foxby}
H.-B. Foxby.
\newblock Isomorphisms between complexes with applications to the homological
  theory of modules.
\newblock {\em Math. Scand.}, 40(1):5--19, 1977.

\bibitem{FIJ}
A.~Frankild, S.~Iyengar, and P.~J{\o}rgensen.
\newblock Dualizing differential graded modules and {G}orenstein differential
  graded algebras.
\newblock {\em J. London Math. Soc. (2)}, 68(2):288--306, 2003.

\bibitem{FJ}
A.~Frankild and P.~J{\o}rgensen.
\newblock Gorenstein differential graded algebras.
\newblock {\em Israel J. Math.}, 135:327--353, 2003.

\bibitem{JorDual}
P.~J{\o}rgensen.
\newblock Recognizing dualizing complexes.
\newblock {\em Fund. Math.}, 176(3):251--259, 2003.

\bibitem{JorAmp}
P.~J{\o}rgensen.
\newblock Amplitude inequalities for differential graded modules.
\newblock {\em Forum Math.}, 22(5):941--948, 2010.

\bibitem{KonSo}
M.~Kontsevich and Y.~Soibelman.
\newblock Notes on {$A_\infty$}-algebras, {$A_\infty$}-categories and
  non-commutative geometry.
\newblock In {\em Homological mirror symmetry}, volume 757 of {\em Lecture
  Notes in Phys.}, pages 153--219. Springer, Berlin, 2009.

\bibitem{Minamoto}
H.~Minamoto.
\newblock Homological identities and dualizing complexes of commutative
  differential graded algebras.
\newblock {\em Israel J. Math.}, 242(1):1--36, 2021.

\bibitem{Mi}
H.~Minamoto.
\newblock Resolutions and homological dimensions of {DG}-modules.
\newblock {\em Israel J. Math.}, 245(1):409--454, 2021.

\bibitem{NT}
T.~Nakamura and P.~Thompson.
\newblock Minimal semi-flat-cotorsion replacements and cosupport.
\newblock {\em J. Algebra}, 562:587--620, 2020.

\bibitem{RG}
M.~Raynaud and L.~Gruson.
\newblock Crit\`eres de platitude et de projectivit\'{e}. {T}echniques de
  ``platification'' d'un module.
\newblock {\em Invent. Math.}, 13:1--89, 1971.

\bibitem{ShHomDim}
L.~Shaul.
\newblock Homological dimensions of local (co)homology over commutative
  {DG}-rings.
\newblock {\em Canad. Math. Bull.}, 61(4):865--877, 2018.

\bibitem{ShINJ}
L.~Shaul.
\newblock Injective {DG}-modules over non-positive {DG}-rings.
\newblock {\em J. Algebra}, 515:102--156, 2018.

\bibitem{Shaul2019}
L.~Shaul.
\newblock Completion and torsion over commutative {DG} rings.
\newblock {\em Israel J. Math.}, 232(2):531--588, 2019.

\bibitem{ShCM}
L.~Shaul.
\newblock The {C}ohen-{M}acaulay property in derived commutative algebra.
\newblock {\em Trans. Amer. Math. Soc.}, 373(9):6095--6138, 2020.

\bibitem{ShKoszul}
L.~Shaul.
\newblock Koszul complexes over {C}ohen-{M}acaulay rings.
\newblock {\em Adv. Math.}, 386:Paper No. 107806, 35, 2021.

\bibitem{ShLoci}
L.~Shaul.
\newblock Open loci results for commutative {DG}-rings.
\newblock {\em J. Pure Appl. Algebra}, 226(5):Paper No. 106922, 11, 2022.

\bibitem{shaul2021sequence}
L.~Shaul.
\newblock Sequence-regular commutative {DG}-rings.
\newblock {\em J. Algebra}, 647:400--435, 2024.

\bibitem{SP}
{The Stacks Project Authors}.
\newblock \textit{Stacks Project}.
\newblock \url{https://stacks.math.columbia.edu}.

\bibitem{YeDual}
A.~Yekutieli.
\newblock Duality and tilting for commutative {DG}-rings.
\newblock arXiv:1312.6411v4, 2016.

\bibitem{Yebook}
A.~Yekutieli.
\newblock {\em Derived categories}, volume 183 of {\em Cambridge Studies in
  Advanced Mathematics}.
\newblock Cambridge University Press, Cambridge, 2020.

\bibitem{Hu}
B.~Zimmermann~Huisgen.
\newblock The finitistic dimension conjectures---a tale of {$3.5$} decades.
\newblock In {\em Abelian groups and modules ({P}adova, 1994)}, volume 343 of
  {\em Math. Appl.}, pages 501--517. Kluwer Acad. Publ., Dordrecht, 1995.

\end{thebibliography}
\end{document}